\documentclass[12pt,reqno]{amsart}

\usepackage{amsfonts,amsthm,amsmath,amssymb,amscd}
\usepackage{mathtools}
\usepackage{esint}
\usepackage{graphics}
\usepackage{indentfirst}
\usepackage{latexsym}
\usepackage[dvips]{epsfig}
\usepackage{mathrsfs}
\usepackage{hyperref}
\usepackage{cite}
\usepackage{color}

\newtheorem{Theorem}{Theorem}
\newtheorem{Corollary}{Corollary}
\newtheorem{Lemma}{Lemma}
\newtheorem{Proposition}{Proposition}
\newtheorem{Definition}{Definition}
\newtheorem{Remark}{Remark}

\newcommand{\ptl}{\partial}

\newcommand{\supp}{\text{supp}}

\begin{document}
	
\def\intco{\text{int conv }}

\def\bs{\boldsymbol}
\def\del{\partial}
\def\DOT{\!\cdot\!}
\def\dt{{\Delta t}}
\def\dx{{\Delta x}}
\def\dy{{\Delta y}}
\def\dz{{\Delta z}}
\def\div{\text{div}}
\def\mb{\mathbf}
\def\c{\mb{c}}
\def\e{\mb{e}}
\def\u{\mb{u}}
\def\m{\mb{m}}
\def\r{\mb{r}}
\def\w{\mb{w}}
\def\U{\mb{U}}
\def\R{\mathbb{R}}
\def\S{\mathbb{S}}
\def\T{\mathbb{T}}
\def\Z{\mathbb{Z}}
\def\mbP{\mathbb{P}}
\def\mcA{\mathcal{A}}
\def\mcC{\mathcal{C}}
\def\mcD{\mathcal{D}}
\def\mcE{\mathcal{E}}
\def\mcF{\mathcal{F}}
\def\mcG{\mathcal{G}}
\def\mcH{\mathcal{H}}
\def\mcI{\mathcal{I}}
\def\mcJ{\mathcal{J}}
\def\mcL{\mathcal{L}}
\def\mcN{\mathcal{N}}
\def\mcO{\mathcal{O}}
\def\mcP{\mathcal{P}}
\def\mcU{\mathcal{U}}
\def\mcV{\mathcal{V}}
\def\mfe{\mathfrak{e}}
\def\mfJ{\mathfrak{J}}
\def\mfO{\mathfrak{O}}
\def\mfQ{\mathfrak{Q}}
\def\msD{\mathscr{D}}
\def\msO{\mathscr{O}}
\def\msQ{\mathscr{Q}}
\def\msR{\mathscr{R}}
\def\brho{\bar{\rho}}
\def\bq{\bar{q}}
\def\Rp{\mathbb{R}_+}
\def\bRp{\bar{\mathbb{R}}_+}

\def\tdz{\tilde{z}}
\def\n{\mb{n}}
\def\V{\mb{V}}
\def\bbn{\bar{\n}}
\def\bV{\bar{\V}}
\def\barm{\bar{\m}}
\def\barU{\bar{\U}}
\def\mfD{\mathfrak{D}}
\def\mfR{\mathfrak{R}}
\def\mcX{\mathcal{X}}
\def\tr{\text{tr}}
\def\x{\mb{x}}
\def\y{\mb{y}}
\def\z{\mb{z}}
\def\g{{\mbox{\boldmath $g$}}}
\def\bu{{\mbox{\boldmath $u$}}}
\def\n{{\mbox{\boldmath $n$}}}
\def\bn{{\mbox{\boldmath $n$}}}
\def\veceps{\mbox{\boldmath $\epsilon$}}
\def\vecdel{\mbox{\boldmath $\delta$}}
\def\vecphi{\mbox{\boldmath $\phi$}}
\def\vecpsi{\mbox{\boldmath $\psi$}}
\def\ds{\mbox{d \boldmath $\gamma$}}
\def\vecomi{\mbox{\boldmath $\theta$}}
\def\eps{\varepsilon}
\def\oneb{{1\mbox \tiny b}}
\def\twob{{2\mbox \tiny b}}
\def\disp{\displaystyle}
\def\dspace{\displaystyle \vspace{.15in}}
\def\half{\textstyle \frac12}
\def\Circ{{\footnotesize$\bigcirc$}}
\def\Triangle{{\small$\triangle$}}
\def\Bullet{{\large$\bullet$}}
\def\exac{{\mbox \tiny e}}
\newcommand{\Dtil}{\widetilde{D}}

\def\TIME{\!\times\!}

\title[Non-Uniqueness of Euler system with Source term]{Non-uniqueness of Admissible Weak Solutions to Compressible Euler Systems with Source Terms}

\author[ ]{Tianwen Luo}
\address{The Institute of Mathematical Sciences and department of
mathematics, The Chinese University of Hong Kong, Hong Kong}
\email{tianwen.luo@foxmail.com}
\author[ ]{Chunjing  Xie}
\address{Department of Mathematics, Institute of Natural Sciences, and Ministry of Education Key Laboratory of Scientific and Engineering Computing, Shanghai Jiao Tong University, 800 Dongchuan Road, Shanghai, 200240, China}
\email{cjxie@sjtu.edu.cn}
\author[ ]{Zhouping Xin}
\address{The Institute of Mathematical Sciences and department of
mathematics, The Chinese University of Hong Kong, Hong Kong}
\email{zpxin@ims.cuhk.edu.hk}

\date{}

\begin{abstract}
	We consider admissible weak solutions to the compressible Euler system with source terms, which include rotating shallow water system and the Euler system with damping as special examples. In the case of anti-symmetric sources such as rotations, for general piecewise Lipschitz initial densities and some suitably constructed initial momentum,  we obtain infinitely many global admissible weak solutions. Furthermore, we  construct a class of  finite-states admissible weak solutions to the Euler system with anti-symmetric sources. Under the additional smallness assumption on the initial densities,   we also obtain  multiple global-in-time admissible weak solutions for more general sources including damping.   The basic framework are based on the convex integration method developed by De~Lellis and Sz\'{e}kelyhidi \cite{dLSz1,dLSz2} for the Euler system. One of the main ingredients of this paper is the construction of specified localized plane wave perturbations which are compatible with a given source term.
\end{abstract}


\maketitle
\section{Introduction}

The well-posedness of the compressible Euler system is one of the central issues in hyperbolic balance laws. The uniqueness of admissible solutions is of fundamental importance. Although there is a quite mature well-posedness theory for one dimensional hyperbolic conservation laws with small BV initial data \cite{Dafermos},  the problem for multi-dimensional systems is  very challenging.  Recently, a major breakthrough for the uniqueness problem is made by De Lellis and Sz\'{e}kelyhidi in  \cite{dLSz1,dLSz2}. Inspired by the surprising examples in \cite{Scheffer} and  \cite{Shnirelman00}, they developed a convex integration framework  and obtained infinitely many bounded weak solutions  to the incompressible Euler system. The convex integration methods are later refined to generate even H\"{o}lder continuous solutions for the incompressible Euler system, see  \cite{dLSz4,Isett12,Buckmaster2013transporting}. The ideas  have also been applied to other systems of PDEs; see \cite{Cordoba11,Sz1,Shvydkoy11,IsettVicol} and the references therein. We refer to \cite{dLSz3} for a general survey on the results in this direction.

 Multiple admissible solutions to the compressible Euler system are obtained in  \cite{Chiodaroli,ChdLKr13,ChKr14,dLSz2,Feireisl14}, by an adaptation of the convex integration method. A major consequence is that the  admissible weak solutions  for the polytropic gases in multidimension are  in general non-unique. These solutions are  called ``wild solutions" in the literature, which reflect the flexibility of the solution space with low regularity and are quite different in nature from those in one dimensional setting such as \cite{Smith79}. It has been shown that many of the available criteria \cite{dLSz2,ChdLKr13}, with the exception of vanishing viscosity limit, are not able to single out a unique solution.

It is interesting to investigate the stability mechanisms which may help to rule out the wild solutions. A natural candidate is the lower order dissipation and dispersion, such as damping and rotating forces. 

In this paper, we consider the compressible Euler system with source terms  as follows
\begin{align}
\begin{cases} \label{eq:CES}
\partial_t \rho  + \nabla \cdot (\rho \u)=0,\\
\partial_t (\rho \u)  + \nabla \cdot (\rho \u \otimes \u) + \nabla p(\rho) = \mb{B} (\rho \u),
\end{cases}
\end{align}
where $(x,t)\in \Omega \times [0,\infty)$ with $\Omega=\R^n$ or $\mathbb{T}^n$; $\rho$, $\u$, and $p$ denote the density, velocity, and the pressure of the flows, respectively. Assume that the equation of states satisfies $p(0)=0$ and $p'(\rho)>0$ for $\rho>0$, and $\mb{B}$ is an $n \times n$ constant matrix.  In particular, the effects of damping and rotating forces are included in this model, where in the case of $n=2$ and nondimensionalization, the matrix $\mb{B}$ has the form $\mb{B} =-\mb{I}$ and $\mb{B}=\mb{J}$, respectively, with
\begin{equation}\label{matrixIJ}
\mb{I}=\begin{pmatrix}
1 & 0\\ 0 & 1
\end{pmatrix}\quad  \text{and}\quad  \mb{J}=\begin{pmatrix}
0 & 1\\ -1 & 0
\end{pmatrix}.
\end{equation}

Most of the previous investigations \cite{Chiodaroli,ChdLKr13,ChKr14,dLSz2,Feireisl14} focused on the isentropic Euler system corresponding to $\mb{B}=\mb{0}$ in the system \eqref{eq:CES}. The existence of infinitely many bounded admissible solutions is first showed by De Lellis and Sz\'{e}kelyhidi in \cite{dLSz2} for a special class of piecewise constant initial densities. The local-in-time existence of multiple admissible solutions for general smooth initial densities  is proved in \cite{Chiodaroli}. When the initial density is close to a constant, the global existence of multiple admissible solutions is obtained in \cite{Feireisl14}. In these works \cite{dLSz2,Chiodaroli,Feireisl14}, the initial momentum are suitably constructed and not explicit. For some classical Riemann initial data connected by shocks, it is proved in  \cite{ChdLKr13,ChKr14} show that the admissible weak solutions for polytropic gases in two-dimension are not unique.  The non-uniqueness issue has also been studied in the class of dissipative weak solutions   for the Euler-Fourier system \cite{ChFeKr13},  the Euler-Korteweg-Poisson system \cite{DoFeMa14}, and the Savage-Hutter model  \cite{FGS15}, respectively, by adaptations of the convex integration method. 

The global well-posedness of classical solutions with damping and rotations has been well-known  under certain smallness assumptions on the initial data; see for examples \cite{ChengXie2011classical,Hsiao1992convergence,Nishida1978nonlinear,Wang2001pointwise}. This is in sharp contrast with the Euler system whose classical solutions  generally develop singularities no matter how smooth and small the initial data are; see for examples  \cite{Sideris,Rammaha}. The main aim of this paper is to investigate whether these lower order dissipations or dispersions  can prevent the loss of uniqueness for admissible weak solutions. Our next aim is to construct wild solutions with special structures in order to gain a better understanding for the fine properties of weak solutions.

 The weak solutions of the system \eqref{eq:CES} are the ones which satisfy \eqref{eq:CES} in the sense of distribution, i.e.,
\begin{Definition}[Weak solutions]  A pair $(\rho,\m)\in L^\infty(\R^n \times [0,\infty);(0,\infty) \times \R^n)$  is said to be a bounded weak solution of the system \eqref{eq:CES} with initial data
	\begin{equation}\label{IC}
	\rho(x, 0)=\rho_0\quad \text{and}\quad \m(x, 0)=\m^\diamond(x),
	\end{equation}
	if for any $ (\varphi,\psi) \in C_c^\infty( \R^n \times [0,\infty); \R \times \R^n)$, it holds that
	\begin{align*}
	&\int_{0}^{\infty}\int_{\R^n} (\rho \partial_t \varphi + \m \cdot \nabla \varphi) dx dt = -\int_{\R^n} \rho_0 \varphi(x,0) dx,\\
	&\int_{0}^{\infty}\int_{\R^n}\left(\m \cdot \partial_t \psi  + \frac{\m \otimes \m}{\rho} : \nabla \psi + p(\rho) \nabla \cdot \psi +  \psi \cdot \mb{B}\m \right)dxdt = -\int_{\R^n} \m^\diamond \cdot \psi(x,0) dx.
	\end{align*}
\end{Definition}
It is well-known that weak solutions  are in general not unique.  Several admissibility criteria to exclude the non-physical solutions have been introduced. An important one is the so called entropy condition. Set
\begin{equation}\label{defmcE}
\mcI(\rho)=\rho \int_0^\rho \frac{p(r)}{r^2}dr.
\end{equation}
It is easy to see that $(\mcI(\rho)+ \frac{|\m|^2}{2\rho}, (\mcI(\rho)+ \frac{|\m|^2}{2\rho}+p(\rho))\frac{\m}{\rho})$ is an entropy-entropy flux pair for the system \eqref{eq:CES} \cite{Dafermos}. Thus one can define  the admissible weak solutions to the system \eqref{eq:CES} as follows.
\begin{Definition}[Admissible weak solutions]
	A bounded weak solution $(\rho,\u)$ is called an  admissible weak solution to \eqref{eq:CES} with initial data \eqref{IC} if for any non-negative test function  $\varphi \in C_c^\infty( \R^n \times [0,\infty) )$, it holds that
	\begin{equation}
	\begin{aligned}
	&\int_{0}^{\infty}\int_{\R^n}\left(\mcI(\rho)+ \frac{|\m|^2}{2\rho}\right)  \partial_t \varphi + \left(\mcI(\rho)+ \frac{|\m|^2}{2\rho}+p(\rho)\right) \frac{\m}{\rho} \cdot \nabla \varphi+\frac{\m}{\rho} \cdot \mb{B}\m \varphi dxdt  \\
	&+\int_{\R^n} \left(\mcI(\rho_0(x))+ \frac{|\m^\diamond(x)|^2}{2\rho_0}\right) \varphi(x,0)dx \geq 0, \label{ineq:energy}
	\end{aligned}
	\end{equation}
	where $\mcI$ is defined in \eqref{defmcE}.
\end{Definition}
In fact, the inequality \eqref{ineq:energy} is exactly the energy inequality for the system \eqref{eq:CES}, and each strong solution to the Cauchy problem \eqref{eq:CES} and \eqref{IC} satisfies \eqref{ineq:energy} with equality. Furthermore, if there is a strong solution for the problem \eqref{eq:CES} and \eqref{IC}, then any admissible weak solutions must coincide with it, see \cite{Dafermos}. 

Now the main results in the paper can be states as follows. We start with the case in which $\mb{B}$ is an anti-symmetric matrix, such as in the rotating shallow water system. For piecewise constant initial densities, the following results hold.

\begin{Theorem}\label{Thm:FiniteValue}
	Suppose that $\mb{B}$ is an anti-symmetric constant matrix.  Assume that $\Omega=\R^n$ or $\mathbb{T}^n$.  Let $\rho_0$ be a piecewise constant positive function in $\Omega$, in the sense that there are a family of at most countably many mutually disjoint open sets $\Omega_i$ with the $n$-dimensional Hausdorff measure $\mcH^n(\Omega \setminus (\cup_i \Omega_i))=0$ and positive constants $\{\bar{\rho}_i\}$ with   $0< \inf_i \bar{\rho}_i  \leq \sup_i \bar{\rho}_i < \infty$, such that
	\[ \rho_0(x) \equiv \bar{\rho}_i \quad \text{ for }x \in \Omega_i.\]  
	Then there exists an $\m^\diamond\in L^\infty(\R^n)$  such that there are infinitely many global bounded admissible weak solutions $(\rho,\m)$ to the Cauchy problem \eqref{eq:CES} and \eqref{IC}.

	Furthermore, either of the following two cases holds: 
	\begin{enumerate}
	 	\item $(\rho,\m)(x,t) = (\brho_i,0) \quad \text{ for a.e. } (x,t) \in \Omega_i \times [0,\infty)$;
	 	\item $(\rho,\m)$ has exactly  $N_n^*$ states in $\Omega_i \times [0,\infty)$, 
	 	where $N_n^*=\frac{n(n+3)}{2}$.
	\end{enumerate}
	There exists at least one $i$ such that the above second case holds in $\Omega_i \times [0,\infty)$.
\end{Theorem}

Several remarks are in order.

\begin{Remark}
	As long as the piecewise constant functions satisfy the Rankine-Hugoniot conditions for the compressible Euler system, they are weak solutions for the system \eqref{eq:CES} associated with $\mb{B}=0$. However, if a piecewise constant function is a weak solution of the rotating shallow water system, it satisfies not only the Rankine-Hugoniot conditions, but also the algebraic system $\mb{B}\m=0$. It is quite surprising that Theorem \ref{Thm:FiniteValue} shows that we can still obtain families of non-trivial finite-states solutions in the class of bounded admissible solutions. This also implies that in the case  $\mb{B}=\mb{J}$ the finite-states weak solution $(\rho,\m)$ can not be continuous at any point in $\Omega_i \times [0,\infty)$ where the second case in Theorem \ref{Thm:FiniteValue} holds.
\end{Remark}

\begin{Remark}
	Theorem \ref{Thm:FiniteValue} is also inspired by the study on
	the problem of finding deformations with finitely many gradients in non-convex calculus of variations (see \cite{Kirchheim2003rigidity}).
\end{Remark}

\begin{Remark}
	For the class of Riemann initial data considered in \cite{ChdLKr13,ChKr14}, we can also show the existence of infinitely many admissible finite-states solutions to \eqref{eq:CES} with $\mb{B}=0$.
\end{Remark}

When the initial density is a general piecewise Lipschitz continuous  function, the following results hold.
\begin{Theorem}\label{Thm:GeneralDensity}
	Suppose that $\mb{B}$ is an anti-symmetric constant matrix. Let $\rho_0$ be a piecewise Lipschitz function in $\R^n$, then there exists an $\m^\diamond\in L^\infty(\R^n)$  such that there are infinitely many global bounded admissible weak solutions $(\rho,\m)$ to the Cauchy problem \eqref{eq:CES} and \eqref{IC}. 
	
	Furthermore, there exists a $T_*>0$ such that each of these  solutions $(\rho,\m)$ reduces to locally finite-states when $t>T_*$, i.e. for a.e. $(x,t) \in \R^n \times (T_*,\infty)$, there exists a neighborhood $\mcN$ of $(x,t)$ such that $(\rho,\m)$ has exactly $N_n^*$ states in $\mcN$.
\end{Theorem}

\begin{Remark}
	The densities of the solutions in Theorem \ref{Thm:GeneralDensity} in general develop discontinuities with complicated geometry even if they are smooth initially. Theorem \ref{Thm:GeneralDensity}  removes the key small-oscillation assumptions on the initial density in \cite{Feireisl14}, where the densities remain smooth for all time.  
\end{Remark}

When $\mb{B}$ is a general matrix, set 
\begin{align}\label{eq:Defb}
\beta = \max\left[0,\sup\{-\xi^{T}\mb{B} \xi:\xi \in \R^{n}, |\xi|=1\}   \right].
\end{align}
Our main results on the non-uniqueness of admissible weak solutions to the Euler system with general source term are as follows.
\begin{Theorem}
	\label{Thm:GeneralSource} Suppose that $0<\check{\rho}  < \hat{\rho}$ are two positive constants. 
	There exists a positive constant $\bar{\eps}$ depending on $\check{\rho},\hat{\rho}$ and $\beta$. Given $\eps\in (0, \bar{\eps})$, if $\rho_0$ satisfies
	\[ 0<\check{\rho} \leq \rho_0 \leq \hat{\rho}\]
	and
	\begin{align}\label{ass:smallness}
	\|(\rho_0-\rho^\sharp,\nabla \rho_0) \|_{L^\infty(\T^n)} \leq \eps
	\quad
	\text{where}\quad
	\rho^\sharp= \dfrac{1}{|\mathbb{T}^n|}\int_{\mathbb{T}^n} \rho_0(x)dx,
	\end{align}
	then there exists an $\m^\diamond\in L^{\infty}(\mathbb{T}^n)$ such that
	the Cauchy problem \eqref{eq:CES} and \eqref{IC} has infinitely many  global-in-time bounded  admissible weak solutions which also satisfy 
	\begin{align}\label{ineq:decay1}
	\|(\rho-\rho^\sharp,\m)\|_{L^\infty(\mathbb{T}^n)} \leq \kappa e^{-\beta t},
	\end{align}
	where $\kappa$ depends on $\eps$ and tends to zero as $\eps$ goes to zero. 
\end{Theorem}

\begin{Remark}
	Theorem \ref{Thm:GeneralSource} is also true  for the Cauchy problem in $\R^n$ provided that there exists a constant $\bar\rho$ such that $\rho_0-\bar\rho$  decays sufficiently  fast at infinity.
\end{Remark}

\begin{Remark}
	A weak solution to the Cauchy problem \eqref{eq:CES} and \eqref{IC} is said to be dissipative if the total energy is non-increasing.
	 Recall that admissible weak solutions  satisfy the local energy inequality \eqref{ineq:energy}, which  implies the non-increasing of the total energy. The admissible criterion \eqref{ineq:energy} is more restrictive than the dissipative criterion. In fact Theorem \ref{Thm:GeneralSource} holds without the smallness assumption \eqref{ass:smallness} in the class of dissipative weak solutions. For the recent studies on uniqueness of dissipative weak solutions, see  \cite{DoFeMa14,FGS15}.
\end{Remark}

We now make some comments on the analysis in this paper. We adapt the convex integration framework in \cite{dLSz2,Chiodaroli} and  reformulate the system \eqref{eq:CES} as a linear system $\mathscr
{L}$ coupled with nonlinear pointwise constrains. Subsolutions are defined by relaxing the constrain sets. In order to obtain finite-states solutions we consider more general relaxed sets than those in \cite{dLSz2,Chiodaroli}. The key for the convex integration scheme \cite{dLSz2,Sz2} is to analyze the wave cone $\Lambda_{\mathscr{L}}$ which generates localized plane waves of the linear system $\mathscr{L}$. Since the plane waves of $\mathscr{L}$ are different when the source terms appear, the results for the Euler system with $\mb{B}=\mb{0}$ in \cite{dLSz1,dLSz2}  do not apply directly.    Our key observation is that in the high frequency regime, the plane wave solutions for the homogeneous system $\mathscr{L}_0$ are good approximations for $\mathscr{L}$. After correcting the errors by solving divergence equations \cite{dLSz4}, we find an elementary method to generate localized plane wave solutions to $\mathscr{L}$  for any constant matrix $\mb{B}$. In the case that $\mb{B}=\mb{0}$, our method also gives an alternative construction for localized plane waves from the ones in \cite{dLSz1,dLSz2}.

 To obtain weak solutions, we use a convex integration scheme to iterate  subsolutions.  The scheme is inspired by \cite{Choffrut,Sz2}. One of the key ingredients of this paper is that we get some suitable relaxed sets with $\frac{n(n+3)}{2}$ extreme points  by a careful perturbation argument using the Carath\'{e}odory's theorem for convex sets, which plays an important role in constructing finite-states weak solutions. When the initial densities have large variations, we construct   non-smooth subsolution ansatz with complicated geometry building on the smooth ansatz in \cite{Feireisl14}. This allows us to remove the  key small-oscillation 
assumption on initial densities in \cite{Feireisl14}, and is the key to proving Theorem \ref{Thm:GeneralDensity}. The major idea to prove Theorem \ref{Thm:GeneralSource} is to construct a strict subsolution ansatz employing the acoustic potential $\Psi$ in \cite{Feireisl14}. 

The rest of the paper is organized as follows. The reformulations of the problem and the concepts of subsolutions  are introduced in  Section \ref{secsub}. We analyze the localized plane waves  in Section \ref{secplane}. In Section \ref{secconint}, weak solutions are obtained by iterating the subsolutions using the localized plane waves. In Section \ref{secadweak}, suitable subsolutions ansatz with various initial data are constructed and the main results are proved.

\subsection*{Notations}
The following notations are used in the rest of the paper.
The bold letters in lower and upper cases are used to denote vectors and matrices, respectively.  Set $\R_+=(0,\infty)$, $\bRp=[0,\infty)$, and $S^{n-1}$ to be the unit sphere of $\R^n$. Denote by $\S_0^{n\times n}$ the vector space consisting of symmetric trace-free $n \times n$ matrices.  Let 
\begin{equation}
N_n=\dim(\R^n \times \S_0^{n \times n})=\frac{n(n+3)}{2}-1, \quad \text{ and } \quad N_n^* =\frac{n(n+3)}{2}. \label{def:N_n}
\end{equation}
Given a set $K \subset \R^n \times \S_0^{n \times n}$, denote by $\text{conv } K$ the convex hull of $K$, and $\intco K$ the interior of $\text{ conv } K$. The space $C^k(\mfD; \mfR)$ is the set of functions mapping from $\mfD$ to $\mfR$ with continuous derivatives up to order $k$ ($k$ can be infinity),  and $C_c^k(\mfD;\mfR)$ is a subset of $C^k(\mfD;\mfR)$ for functions with compact support. Suppose $f \in C^0(\mfD;\mfR)$, let 
\begin{align}
\text{image}(f)=\{\mathfrak{r}  \in \mfR: \mathfrak{r}=f(\mathfrak{d})\text{ for some } \mathfrak{d} \in \mfD \} \label{def:image}
\end{align}
denote the image of $f$. A function $f \in L^\infty(\Rp \times \R^n)$ is said to be in  $C_{loc}(\bRp;L^\infty_{w*}(\R^n))$, if for any test function $\phi \in L^1(\R^n)$ and any compact subset $J \subset \bRp$,
\begin{equation}\label{eq:XTopo}
\int_{\R^n}f\phi dx \to \int_{\R^n}f\phi dx \quad  \text{uniformly on } J.
\end{equation} Denote by
\[
\mcL(\mfD) =\{f| f\in L^1(\mfD), \int_\mfD f=0\}.
\]
 Let $\tr(\mb{A})$  denote the trace of the matrix $\mb{A}$. Two matrices $\mb{A}$ and $\mb{B}$ are said to be $\mb{A}\leq \mb{B}$ (or $\mb{A}< \mb{B}$) if for any $\xi\in S^{n-1}$, one has
\begin{equation*}
\xi^T\mb{A}\xi\leq \xi^T\mb{B}\xi \quad (\text{or}\,\,\xi^T\mb{A}\xi< \xi^T\mb{B}\xi).
\end{equation*}
Suppose that $A$, $B\in Y$ which is a linear space, then  $[A, B]$ is called the line segment connecting $A$ and $B$, i.e.,
\[
[A, B]=\{\theta A+(1-\theta)B|\theta \in [0, 1]\}.
\]
Suppose $M,N$ are two subsets in a vector space $Y$. Define
\[ M+N = \{y \in Y  : y = y' + y'', y' \in M, y'' \in N. \} \]
Given a space-time set $\mcD\subset \R^{n+1}$, then $\mcD_t =\{x| (x, t)\in \mcD\}$ is the projection of $\mcD$ on $\R^n$ and $\mcD^{t}=\mcD\setminus (\overline{\mcD_t}\times \{t\})$.
Let $\bar z=(\bar x,\bar t)$ be a point in  $\R^{n+1}$ and denote by $\msQ_r(\bar x)$ and $Q_r(\bar{z})= Q_r(\bar{x},\bar t)$ the open space cube and space-time cube, respectively, i.e.
\begin{align}\label{def:msQ}
\msQ_r(\bar{x})=\{x \in  \R^n: |x_i-\bar{x}_i|<\frac{r}{2}, i=1,\cdots,n\}\,\,
\end{align}
and
\begin{equation}\label{def:Q}
Q_r(\bar{z})=\{z \in \R \times \R^n: |t-\bar{t}| < \frac{r}{2}, |x_i-\bar{x}_i|<\frac{r}{2},i=1,\cdots,n\}.
\end{equation}
For simplicity, denote $Q_r=Q_r(0)$. Denote by  $\mcH^k$ the $k$-dimensional Hausdorff measure. Given a set $\mcO \subset \R^k$, we also use $|\mcO|$ to denote $\mathcal{H}^k(\mcO)$ when there is no ambiguity.

\section{Subsolutions}\label{secsub}
In this section, the system \eqref{eq:CES} is reformulated as a differential inclusion in the spirit of De Lellis and Sz{\'e}kelyhidi \cite{dLSz2,Chiodaroli}. Then subsolutions as relaxed inclusions  are introduced. However, the formulation here allows for more general constrain sets compared with \cite{dLSz2,Chiodaroli}, which will be employed to obtain finite-states weak solutions. Our formulation is inspired by \cite{CGSW14,Choffrut,Sz1}.

Given $\rho>0$ and $q\geq 0$, set  
\[ K_{\rho,q}=\{(\m,\U)\in \R^n \times \S_0^{n \times n}:\frac{\m \otimes \m}{\rho}-\U = qI\}.\]
 By taking traces, it is easy to see that
\begin{equation}
K_{\rho,q}=\{(\m,\U)\in \R^n \times \S_0^{n \times n}:|\m|^2 = n\rho q, \U=\frac{\m \otimes \m}{\rho}-qI\}. \label{eq:eqDefK}
\end{equation} 
Therefore, $K_{\rho,q}$ is bounded in $\R^n \times \S_0^{n \times n}$. We denote 
\begin{equation}
 \mathfrak{C}(\rho,q) = \sup_{w \in \text{ conv }K_{\rho,q}} |w| < \infty. \label{ineq:Kbdd}
\end{equation}

 Observe that a bounded weak solution $(\rho,\mb{m})$ to the system \eqref{eq:CES} is equivalent to a bounded quadruple $(\rho,\mb{m},\mb{U},q)$  satisfying the differential equations
 \begin{align}
 &\ptl_t \rho + \nabla \cdot \m  = 0, \label{eq:subsolutonPDE1}\\
 &\partial_t \m + \nabla \cdot \U + \nabla\left(p(\rho)+q\right) = \mb{B}\m, \label{eq:subsolutonPDE2}
 \end{align}
 in the sense of distribution and the constrains
\begin{equation}
(\m,\U)(x,t) \in K_{\rho(x,t),q(x,t)} \quad a.e. 
\end{equation}
The subsolutions are defined by relaxing the constrains as follows.
\begin{Definition}\label{def:subsolu}
	The quadruple $(\rho,\m, \U, q) \in L^\infty(\R^n \times \bRp;\Rp \times \R^n \times \S_0^{n\times n} \times \bRp)$ is called a subsolution of the system (\ref{eq:CES}) if it solves \eqref{eq:subsolutonPDE1} and \eqref{eq:subsolutonPDE2} in the sense of distribution  and satisfies the relaxed constrains 
	\begin{align}
	(\mb{m},\mb{U})(x,t) \in \text{ conv } K_{\rho(x,t),q(x,t)} \quad  a.e.\,\, \text{in}\, \, \R^n\times \bRp. 
	\end{align}
	Furthermore, 	given   a family of compact sets $K_{(x,t)} \subset K_{\rho(x,t),q(x,t)}$ such that the map $(x,t) \mapsto K_{(x,t)}$ is  continuous in  a space-time open set $\mcD$ in the Hausdorff distance (\cite[Definition 7.3.1.]{BBI}), a subsolution $(\rho,\m,\U,q)$ is said to be strict in $\mathcal{D}$ with constrain sets $K_{(x,t)}$ 
	 if \[ (\rho,\m,\U,q) \in C_{loc}(\bRp;L^\infty_{w*}(R^n)), \quad (\rho,\m,\U,q)\big|_\mcD \in C^0(\mcD),\] 
	 and
	\begin{equation}
	(\mb{m},\mb{U})(x,t) \in \intco K_{(x,t)} \text{ in } \mathcal{D}, \quad \rho > 0 \text{ and } q > 0 \text{ in } \mcD. \label{ineq:ss}
	\end{equation}
\end{Definition}
Notice that the strict subsolutions defined in \cite{dLSz2,Chiodaroli} correspond to the cases that the constrain sets $K_{(x,t)}=K_{\rho(x,t),q(x,t)}$ in Definition \ref{def:subsolu}.   

The following lemma, appeared in \cite{dLSz2}, plays an important role in constructing strict subsolutions, and implies in particular that  $(0,0) \in \intco K_{\brho,\bq}$, for $\brho,\bq>0$. 
\begin{Lemma}\label{Lemma:StateSpace}
	The convex hull of $K_{\brho, \bq}$ is characterized as:
	\begin{align}
	\text{ conv }K_{\brho,\bq}= \{(\bbn,\bV) \in \R^n\times \S_0^{n\times n}: \bbn \otimes \bbn - \brho \bV \leq \brho \bq \mb{I} \}. \label{eq:coK}
	\end{align}
	Consequently, 
	\begin{align}
	\intco K_{\brho,\bq}= \{(\bbn,\bV) \in \R^n\times \S_0^{n\times n}: \bbn \otimes \bbn - \brho \bV < \brho \bq \mb{I} \}.\label{ineq:SubsolutionNonstrict}
	\end{align}
\end{Lemma}
\begin{proof}
	In the case of $K_{1,1/n}$, the characterization \eqref{eq:coK} is obtained in \cite{dLSz2}. The general case follows from the linear transform 
	\begin{equation}\label{eq:LT}
	(\m,\U) \mapsto \left(\frac{\m}{\sqrt{n\rho q}},\frac{1}{nq}\U
	\right), \quad \text{ for } (\m,\U) \in \R^n\times \S_0^{n\times n},
	\end{equation}
	which maps $K_{\rho,q}$ one-to-one to $K_{1,1/n}$, and preserves convex hulls.
%
\end{proof}

In the rest of this section, we consider convex sets with finite extreme points, which will be employed for constructing finite-states weak solutions.
First, the following elementary result holds.
\begin{Lemma}\label{Lemma:FSC}
	Suppose that $K$ is a compact set in $\R^s$. 
	Let $\mcC$ be a compact set such that $\mcC \subset \intco K$. Then there exists a finite set $\tilde{K}$ of $K$ such that $\mcC \subset \intco \tilde{K}$.
\end{Lemma}

\begin{proof}
	 For $w \in \mcC$, there exists  a simplex $\mathcal{S}$  such that $w \in \text{int }\mathcal{S}$ and $\mathcal{S} \subset \intco K$. Let $\{\mathcal{S}_i\}$ be a finite sequence of simplexes which cover the compact set $\mcC$.  It follows from Carath\'{e}odory's theorem  that   each vertex $\mathcal{W}_{i,j}$ of $\mathcal{S}_i$ can be represented as convex combinations of at most $s+1$ extreme points $\{\mcV_{i,j,k}\}_{k=1}^{s+1}$ of $\text{ conv } K$. Since $K$ is compact, the extreme points of $\text{ conv } K$ must belong to $K$.  Let
	$\tilde{K}$ be the collection of these extreme points $\{\mcV_{i,j,k}\}$. Clearly that $ \mcC \subset \cup_i \text{int } \mathcal{S}_i \subset \intco \tilde{K}$. 
\end{proof}
In the case that  $\mcC$ consists of a single point, the following sharp result holds. 
\begin{Lemma}\label{Lemma:FiniteValue}
	Let $\bar{w} \in \intco K_{\brho, \bq}$. Then there exists a finite set of $N_n^*$ points $\mathcal{K} \subset  K_{\brho, \bq}$   such that $\bar{w} \in \intco \mathcal{K}$, where $N_n^*=N_n+1=\frac{n(n+3)}{2}$ as defined in \eqref{def:N_n}.
\end{Lemma}

\begin{proof}
	It follows from Lemma \ref{Lemma:FSC} that there exists a finite set $F=\{w_i \}_{i = 1}^N \subset K_{\brho, \bq}$ such that $ \bar{w} \in \intco F$. Using Carath\'{e}odory's theorem yields that there is a subset $\{w_{i_j}\}_{j=1}^{N_n^*}$ with $\bar{w} \in \text{conv } \{w_{i_j}\}_{j=1}^{N_n^*}$. However, it may happen that $\bar{w}  \notin \intco \{w_{i_j}\}_{j=1}^{N_n^*}$. The strategy is to avoid the degenerate case by perturbing the vertices  slightly.  An $N_n$-set $G=\{v_j\}_{j=1}^{N_n} \subset K_{\brho, \bq}$ is said to be non-degenerate if 		
	\begin{equation}\label{eq:LI}
	\{v_j-\bar{w}\}_{j=1}^{N_n} \text{ is a set of linearly independent vectors in }  \R^n \times \S_0^{n \times n}.
	\end{equation}
	 The number of distinct $N_n$-subsets of $F$ is $J={N \choose N_n}=\frac{N!}{N_n!(N-N_n)!}$.  Suppose that for some $k$ with  $0 \leq k < J$, the set $F_k=\{w_i^{(k)}\}_{i = 1}^N\subset K_{\brho, \bq}$ has been obtained with the following two properties: 
	\begin{enumerate}
		\item there are at least $k$ non-degenerate $N_n$-subsets of $F_k$;
		\item it holds that $\bar{w} \in \intco F_{k}$.
	\end{enumerate}
	For $k=0$, the set $F_0 = F$ clearly satisfies the above two properties. Suppose that we have obtained $F_0, \cdots, F_k$ which satisfy the above two properties. Let $G_1, \cdots, G_l$ be all the non-degenerate $N_n$-subsets of $F_k$. Set $F_{l}=F_k$ if $l>k$. Otherwise the number of non-degenerate $N_n$-subsets of $F_k$ is exactly $k$. Since $k< {N \choose N_n}$, there is a degenerate $N_n$-subset $G=\{v_i\}_{i=1}^{N_n}$   of $F_k$, i.e. the vectors $\{v_i-\bar{w}\}_{i=1}^{N_n}$ are linearly dependent. Consider a perturbation $\tilde{G}=\{\tilde{v}_i\}_{i=1}^{N_n} \subset  K_{\brho, \bq}$ of $G$. It follows from  a continuity argument that there exists a $\delta_0>0$ such that  $\bar{w} \in \intco ((F_k \setminus G)\cup \tilde{G})$ and that the $N_n$-subsets $\tilde{G}_j = (G_j \setminus G) \cup \tilde{G}$ are non-degenerate for $1 \leq j \leq k$, whenever 
	\begin{equation}
	\max_{i=1}^{N_n}|\tilde{v}_i-v_i|<\delta_0. \label{eq:v_i}
	\end{equation} 
	In view of Lemma \ref{Lemma:DimensionCounting} below, it is not hard to see that there exists a  set $\tilde{G}$  which  is non-degenerate  and satisfies \eqref{eq:v_i}.  Set $F_{k+1}=(F_k \setminus G)\cup \tilde{G}$. It is easy to see that $F_{k+1} \subset  K_{\brho, \bq}$ and satisfies the two properties for $k+1$. After at most ${N \choose N_n}$ steps, we obtain $\tilde{F}=F_{J}$ such that $\bar{w} \in \intco \tilde{F}$ and all the $N_n$-subsets of $\tilde{F}$ are non-degenerate.

	It follows from Carath\'{e}odory's theorem that there exists a $N_n^*$-subset $\mathcal{K}$ of $\tilde{F}$ such that $\bar{w} \in \text{conv } \mathcal{K}$. Furthermore, since any  $N_n$-subsets of $\tilde{F}$ is non-degenerate, one has $\bar{w} \in \intco \mathcal{K}$.
\end{proof}
%
\begin{Lemma}\label{Lemma:DimensionCounting}
	Suppose $\bar{\rho}>0$ and $\bar{q}>0$. Given $w_0 \in K_{\bar{\rho},\bar{q}}$,  for any given $\delta>0$, the set \[\mcN_{\delta}(w_0):=\{ w \in K_{\bar{\rho},\bar{q}}: |w-w_0|< \delta\}\] does not lie in any hyperplane $\mcP \subset \R^n \times \S_0^{n\times n}$.
	
\end{Lemma}

\begin{proof}
	First consider the case that $\bar{\rho}=1,\bar{q}=1/n$, i.e. $ K_{\bar{\rho},\bar{q}}=K_{1,1/n}$. For a subset $M \subset \R^n \times \S_0^{n \times n}$, let $\text{affdim}(M)$ denote the dimension of its affine hull in $\R^n \times \S_0^{n \times n}$, i.e. the intersection of all affine sets containing $M$. For $w \in K_{1,1/n}$, let 
	\[ \mathfrak{A}(w) = \inf_{\delta>0}(\text{affdim}(\mcN_{\delta}(w))).\]
	It is easy to see  that $K_{1,1/n}$ is diffeomorphic to $S^{n-1}$ via the map 
	\[  \xi \mapsto  (\xi,\xi \otimes \xi - \frac{1}{n}\mb{I}) \in K_{1,1/n} \quad \text{ for } \xi \in S^{n-1}.\]
	Let $T_{w}(K_{1,1/n})$ denote the tangent space of $K_{1,1/n}$ at $w$.  It is easy to see that 
	\[ n-1=\dim(T_{w}(K_{1,1/n}))\leq \mathfrak{A}(w) \leq N_n, \quad \text{ for } w \in K_{1,1/n}.\]
	We claim that $\mathfrak{A}(w)$ equals to some positive constant $N_0$ on $K_{1,1/n}$, due to homogeneity. To see this,  consider the action of $SO(n)$ in $\R^n \times \S_0^{n \times n}$ inspired by \cite{Choffrut}. It is easy to see that for any  $\mb{R} \in SO(n)$, the linear transform on $\R^n \times \S_0^{n \times n}$ defined by
	\[ \msR(\m,\U) = (\mb{R}\m,\mb{R}\U \mb{R}^t) \quad \text{ for }(\m,\U)\in \R^n \times \S_0^{n \times n},\]
	preserves affine hulls and transforms $K_{1,1/n}$ into itself. Then it holds that $\mathfrak{A}(\mb{R}w) = \mathfrak{A}(w)$ for $w \in K_{1,1/n}$. Suppose $w_1$ and $w_2$ are any two points of  $K_{1,1/n}$ which are of the forms
	\[ w_i=(\xi_i,\xi_i \otimes \xi_i - \frac{1}{n}\mb{I}), \quad \xi_i \in S^{n-1},i=1,2.\]  
	There exists a rotation $\mb{R} \in SO(n)$ such that $\xi_2 = \mb{R}\xi_1$. It follows that $w_2 = \msR w_1$  and thus $\mathfrak{A}(w_2)=\mathfrak{A}(w_1)$. 
	
	 We now show that  $N_0 = N_n$. Suppose $N_0 < N_n$. Fix some $w_0 \in K_{1,1/n}$. Then there exist a $\delta_0>0$ and an affine subspace $L_0$ of dimension $N_0$ such that $\mcN_{\delta_0}(w_0)\subset L_0$. Recall from  Lemma \ref{Lemma:StateSpace} that $(0,0) \in \intco K_{1,1/n}$, hence the affine dimension of $K_{1,1/n}$ is $N_n$, which is larger than $\dim(L_0)=N_0$. It follows that the number
	\[\bar{\delta} = \sup\{\delta>0: \mcN_{\delta}(w_0)\subset L_0 \}  \]
	is finite.
	From the definition of $\bar{\delta}$ and the compactness of $\ptl \mcN_{\bar{\delta} }(w_0) \subset K_{1,1/n}$, there exists a $w_1 \in \ptl \mcN_{\bar{\delta} }(w_0)$ with the property that $\mcN_{\delta}(w_1)$ is not contained in $L_0$ for any $\delta>0$. Since $\mathfrak{A}(w_1) = N_0$, there exist a $\delta_1>0$ and an affine subspace $L_1$ of dimension $N_0$ such that $\mcN_{\delta_1}(w_1)\subset L_1$. It is easy to see that $L_0 \neq L_1$ and thus $\dim(L_0 \cap L_1)<N_0$. Observe that for $\delta<\bar{\delta}$ which is sufficiently close to $\bar{\delta}$, the intersection $\mcN_{\delta }(w_0)\cap \mcN_{\delta_1}(w_1)$ is a nonempty open set of $K_{1,1/n}$. Therefore, for $w \in \mcN_{\delta }(w_0)\cap \mcN_{\delta_1}(w_1)  \subset L_0 \cap L_1$, it holds that $\mathfrak{A}(w)\leq \dim(L_0 \cap L_1)< N_0$. This contradicts with the fact that $\mathfrak{A}(\cdot)$ is constant on $K_{1,1/n}$. It follows that $N_0 = N_n$ and the case for $K_{1,1/n}$ is proved. 
	
	The general case for $K_{\bar{\rho},\bar{q}}$ follows
	from the linear transform \eqref{eq:LT}
	which maps $K_{\bar{\rho},\bar{q}}$ one-to-one to $K_{1,1/n}$ and preserves hyperplanes. This finishes the proof of the lemma.
\end{proof}


\section{Localized Plane Waves}\label{secplane}
In this section, localized plane waves  are constructed, which form the building blocks for the iteration scheme to obtain weak solutions. 


Define the linear operator
\begin{equation}
\mathscr{L}(\n,\V) =(\nabla \cdot \n, \ptl_t \n + \nabla \cdot \V - \mb{B}\n)^t.
\end{equation}
Observe that if $(\rho,\m,\U,q)$  solves \eqref{eq:subsolutonPDE1} and \eqref{eq:subsolutonPDE2} and $(\n,\V)$ satisfies
\begin{equation}
\mathscr{L}(\n,\V)=0, \label{msL}
\end{equation} 
then $(\rho,\m+\n,\U+\V,q)$ also solves \eqref{eq:subsolutonPDE1} and \eqref{eq:subsolutonPDE2}. The convex integration scheme will employ localized plane waves of $\mathscr{L}$. It needs to be shown  that the wave cone $\Lambda_\mathscr{L}$  defined below (according to \cite{Sz1}) is suitably large.
\begin{Definition}
	The wave cone $\Lambda_\mathscr{L}$ associated with $\mathscr{L}$ is a subset of  $\mathbb{R}^n \times \mathbb{S}_0^{n \times n}$, to which there exists a constant $C>0$ such that for any $(\bar{\mb{n}},\bar{\mb{V}}) \in \Lambda_\mathscr{L}$  there exists a sequence $(\tilde{\n}_k,\tilde{\V}_k) \in C_c^\infty(Q_1;\mathbb{R}^n \times \mathbb{S}_0^{n \times n})$ solving $\mathscr{L}(\tilde{\n}_k,\tilde{\V}_k)=0$  such that
	\begin{itemize}
		\item $\text{dist}((\tilde{\n}_k,\tilde{\V}_k),[-(\bbn,\bV),(\bbn,\bV)]) \to 0$ uniformly in $Q_1$ as $k \to \infty$,
		\item $(\tilde{\n}_k,\tilde{\V}_k) \to 0$ in the sense of distribution as $k \to \infty$,
		\item $\int_{Q_1} |(\tilde{\n}_k,\tilde{\V}_k)|^2(x,t) dx dt  \geq C |(\bbn,\bV)|^2$ for all $k \in \mathbb{N}$.
	\end{itemize}
	The functions $(\tilde{\n}_k,\tilde{\V}_k)$ are called localized plane waves of $\mathscr{L}$.
\end{Definition}
Let  $\mathscr{L}_0$ denote the homogeneous operator corresponding to $\mb{B}=0$ and $\Lambda_{\mathscr{L}_0}$ the associated wave cone. It is proved in \cite{dLSz2} that if $w_1 \neq w_2$, for any $w_i \in K_{\bar{\rho},\bar{q}}$ with $\bar{\rho}>0$ and $\bar{q}>0$,
\begin{equation}\label{eq:diffLambda}
w_1-w_2 \in \Lambda_{\mathscr{L}_0}.
\end{equation}
The proof of \eqref{eq:diffLambda} employs the key property that $w_1$ and $w_2$ is connected by a plane wave solution to $\mathscr{L}_0(\n,\V) = 0$ as follows. 
\begin{Lemma}\label{Lemma:AD}
	Let $\Lambda$ denote the subset of $\R^n \times \S_0^{n \times n}$ such that for $(\bbn,\bV) \in \Lambda$ there exists a $(\tau,\xi)\in \R \times S^{n-1}$ such that  $(\bbn,\bV)h(\tau t + \xi \cdot x )$ is a plane wave solution to $\mathscr{L}_0(\n,\V) = 0$ for any smooth function $h$. Then for $w'$ and $w'' \in K_{\brho,\bq}$ with $w' \neq w''$
	\[ w'-w'' \in \Lambda.\]
\end{Lemma}
\begin{proof}
	For $(\bbn',\bV')$, $(\bbn'',\bV'') \in K_{\brho,\bq}$, set $(\bbn,\bV)=(\bbn',\bV')-(\bbn'',\bV'')$. Let $\xi \in S^{n-1}$  satisfy $\xi \cdot \bbn =0$ so that $\xi\cdot \bbn' =\xi\cdot \bbn''$. Denote $\tau = - (1/\brho)\xi \cdot \bbn'$. Then one has
	\begin{align*}\label{eq:AdmissibleDifference}
	\tau  \bbn +  \bV \cdot \xi= \tau  \bbn + \frac{1}{\brho}(\bbn' \otimes \bbn' - \bbn'' \otimes \bbn'')\cdot \xi =  (\tau + \frac{1}{\brho}\xi \cdot \bbn')(\bbn' - \bbn'')=0.
	\end{align*}
	This finishes the proof of the lemma.
\end{proof}
However, when the source $\mb{B}$ is non-zero, for $w_1, w_2 \in K_{\bar{\rho},\bar{q}}$  there may not be plane wave solutions to $\mathscr{L}$ with profiles $w_1 - w_2$. It seems initially unclear whether the property \eqref{eq:diffLambda} is still true for $\Lambda_{\mathscr{L}}$ in the presence of the source terms. Our key observation is that: in the high-frequency regime, localized plane waves  with sources can be constructed as perturbations of the plane waves of $\mathscr{L}_0$. This is stated as follows. 
\begin{Lemma}[Localized plane waves]\label{Lemma:LPW}
	Suppose that $w,w_1,w_2 \in \R^n \times \S_0^{n \times n}$ satisfy
	\[w= \mu_1 w_1 + \mu_2 w_2, \,\,  \mu_1+\mu_2 = 1, \mu_i>0,\quad \text{ and } \quad \bar{w} = w_2 - w_1 \in \Lambda.\] 
	Given an open space-time set $\mcO$ and $\varepsilon>0$,  there exists a $\tilde{w} \in  C_c^\infty(\mcO;\R^n \times \S_0^{n \times n})$ satisfying the following properties
	\begin{enumerate}
		\item $\tilde{w}(\cdot,t) \in \mcL(\R^n)$ and $\tilde{w}$ solves the equation \eqref{msL}, i.e.
		\begin{equation}
		\int \tilde{w}(x,t)dx = 0 \quad \text{ and } \quad  \mathscr{L}\tilde{w}=0; \label{eq:Lemma6-P1}
		\end{equation}
		\item $\text{dist}(w+\tilde{w}(x,t),[w_1,w_2])<\eps \text{ for } (x,t) \in \mcO$;
		\item there exist two disjoint open sets $\mcO_i \subset \mcO$  such that for $i=1,2,$
		\[|w+\tilde{w}(x,t)-w_i| < \eps  \text{ for } (x,t) \in \mcO_i; \quad |\mcH^{n+1}(\mcO_i) - \mu_i\mcH^{n+1}(\mcO)|<\varepsilon. \]  
	\end{enumerate}
\end{Lemma}
The following characterization of $\Lambda_{\mathscr{L}}$ is a consequence of Lemma \ref{Lemma:LPW}.
\begin{Corollary}\label{Cor:Lambda_L}
$ \Lambda \subset \Lambda_{\mathscr{L}} $.
\end{Corollary}
\begin{proof}[Proof of Corollary \ref{Cor:Lambda_L}]
	Let $\bar{w}= (\bbn,\bV) \in \Lambda$. For each $k \in \mathbb{N}$, applying Lemma \ref{Lemma:LPW} to $w=0=\frac{1}{2}(-\bar{w}) + \frac{1}{2}\bar{w}$ with $\eps=\frac{1}{k}$ yields that there exists $v_k \in C_c^\infty(Q_1)$ satisfying \eqref{eq:Lemma6-P1}  and 
	\[\text{dist}(v_k,[-\bar{w},\bar{w}]) \to 0, \quad  \int_{Q_1} |v_k|^2 dxdt \to |\bar{w}|^2,  \quad \text{ as } k \to \infty. \]
	Let $\{Q_{2^{-k}}(z_j^{(k)})\}_{j=1}^{2^{k(n+1)}}$ be the decomposition of $Q_1$ into mutually disjoint cubes of length $2^{-k}$. Let 
	\[ \tilde{w}_k =  \sum_{j=1}^{2^{k(n+1)}} a_j^{(k)}\quad  \text{ where } a_j^{(k)}(z) = v_k\left(\frac{z-z_j^{(k)}}{2^{-k}}\right) \in C_c^\infty(Q_{2^{-k}}(z_j^{(k)})).\]
	It is clear that $\mathscr{L}\tilde{w}_k = 0$ and $\text{dist}(\tilde{w}_k,[-\bar{w},\bar{w}]) \to 0$ as $k\to \infty$. Since $a_j^{(k)}(\cdot,t) \in \mcL(\R^n)$ and $\text{diam}(\supp a_j^{(k)}) \leq \sqrt{(n+1)}2^{-k} \to 0$, one has  $\tilde{w}_k \to 0$ in distribution. Notice that 
	\[ \int_{Q_1} |\tilde{w}_k|^2  = \sum_{j=1}^{2^{k(n+1)}} \int |a_j^{(k)}|^2 = \int_{Q_1} |v_k|^2. \]
	Therefore, for sufficiently large $k$, one has
	\[ \int_{Q_1} |\tilde{w}_k|^2 dx dt  \geq \frac{1}{2}|\bar{w}|^2. \]
	This shows that $\bar{w} \in \Lambda_{\mathscr{L}} $ and finishes the proof of the corollary.
\end{proof}
The rest of the section is devoted to the proof of Lemma \ref{Lemma:LPW}.

To correct the errors from the source terms, we recall the linear operators which solve divergence equations in $C^\infty(\T^n;\S_0^{n\times n})$  (\cite{dLSz4}) and  adapt them  to the whole spaces.

\begin{Lemma}\cite[Proposition 5.1]{dLSz4}\label{Lemma:OpRTorus} 
	There exists a linear operator $\mathcal{R}_{\T^n}$ from  $C^\infty(\T^n;\R^n) $ to $C^\infty(\T^n; \S_0^{n\times n})$ such that for any $\mb{f} \in C^\infty(\T^n;\R^n)$, it holds that
	\begin{align*}
		\nabla \cdot \mathcal{R}_{\T^n}[\mb{f}]= \mb{f}- \fint_{\T^n} \mb{f},
	\end{align*}
	and
	\begin{align*}
		\| \mathcal{R}_{\T^n}[\mb{f}] \|_{C^{1,\alpha}(\T^n)} \leq C(n,\alpha) \|\mb{f}\|_{C^{\alpha}(\T^n)}.
	\end{align*}
\end{Lemma}
We also need the following lemma, whose proof  is elementary and will be postponed to the end of the section.
\begin{Lemma}\label{Lemma:OpR}
	For any $\mb{f} \in C_c^\infty(\R^n;\R^n)$, there exists an $\mathcal{R}[\mb{f}]\in C^\infty(\R^n; \S_0^{n\times n})$ satisfying
	$$\nabla \cdot \mathcal{R}[\mb{f}]= \mb{f}.$$
	Furthermore, $\mathcal{R}$ satisfies the following properties:
	\begin{enumerate}
		\item $\mathcal{R}$ is a linear operator from $C_c^\infty(\R^n; \R^n)$ to $C^\infty(\R^n; \S_0^{n\times n})$;
		\item $\mathcal{R}[\Delta^2 \mb{f}]$ is a linear combination of third order derivatives of $\mb{f}$;
		\item $\supp \mathcal{R}[\Delta^2 \mb{f}] \subset \supp \mb{f}$ and $\mathcal{R}[\Delta^2 \mb{f}]\in \mathcal{L}(\R^n)$, i.e.,
		\[
		\int_{\R^n} \mathcal{R}[\Delta^2 \mb{f}]dx = 0;
		\]
		\item there exists a constant $0<\alpha<1$ such that 
		\[ \|\mathcal{R}[\mb{f}]\|_{C^\alpha(\R^n)} \leq C \max(\|\mb{f}\|_{L^1(\R^n)},\|\mb{f}\|_{L^\infty(\R^n)}), \]
		where the constants $C$ and $\alpha$  depend only on the dimension $n$.
	\end{enumerate}
\end{Lemma}

\begin{proof}[Proof of Lemma \ref{Lemma:LPW}]
	Let $\delta$ be a small positive constant to be chosen later. Fix a smooth cut-off function $\phi \in C^\infty_c(\mcO)$ such that
	\begin{align}
		0 \leq \phi \leq 1 \quad \text{ and } \quad \mcH^{n+1}\left(\{(x,t)\in \mcO:\phi(x,t) \neq 1 \}\right)<\delta. \label{phi}
	\end{align}
	Set
	\[ h(s) = \begin{cases}
	-\mu_2,\qquad &s \in (0,\mu_1],\\
	\mu_1, \qquad &s \in (\mu_1,1],
	\end{cases} \]
	and extend $h$ as a periodic function with period $1$. Notice that $\int_{0}^{1}h(s)ds = 0$. Let $h_0$ be a  $1$-periodic smooth approximation of $h$ such that 
	\begin{equation}
	 -\mu_2 \leq h_0 \leq \mu_1,\,\, \mcH^1(\{s \in [0,1]:h(s)\neq h_0(s)\}) < \delta,\,\, \int_{0}^1 h_0(s)ds = 0. \label{h_0}
	\end{equation}
For $k=0,1,\cdots$, let 
\[\tilde{h}_{k+1}(s)= \int_{0}^{s}h_k(\sigma)d\sigma, \quad h_{k+1}(s) = \tilde{h}_{k+1}(s) - \int_{0}^{1}\tilde{h}_{k+1}(\sigma)d\sigma. \] 
Thus for any $k \geq 0$, one has
\[ h_{k}^{(j)}(s) = h_{k-j}(s) \quad \text{ for } 0 \leq j \leq k, \quad \int_{0}^{1} h_k(s)ds = 0, \]
and 
\[ \|h_k\|_{L^\infty} \leq \|h_{k-1}\|_{L^\infty} \leq \cdots \leq \|h_0\|_{L^\infty}. \]
	
	Denote $\bar{w} = (\bar{\n},\bar{\V})$ and let $(\tau,\xi) \in \R \times S^{n-1}$ be associated with $\bar{w} \in \Lambda$. Set
	\begin{equation}\label{eq:defn'V'}
		\n'=\lambda^{-6}\Delta^3 [\bar{\n}h_6(\lambda \tau t +\lambda \xi \cdot x) \phi],\quad  \V'=\lambda^{-6}\Delta^3 [\bar{\V}h_6(\lambda \tau t +\lambda \xi \cdot x) \phi],
	\end{equation}
	and
	\begin{equation*}
		\n''= -\nabla \Delta^{-1} \nabla \cdot \n', \quad   \V''=\mathcal{R}[\mb{B} (\n'+\n'') - \partial_t (\n'+\n'') - \nabla \cdot \V'],
	\end{equation*}
	where $\mathcal{R}$ and $\Delta^{-1}\u:=\mcN * \u$ are the operators defined in the proof of Lemma \ref{Lemma:OpR}, and $\lambda \in \R$ is a large positive constant to be determined later. Define \[\tilde{w}=(\n,\V)=(\n'+\n'',\V'+\V'').\]
	It will be verified below  that \eqref{eq:Lemma6-P1} holds, $\supp \tilde{w} \subset \supp \phi$ and 
	\begin{equation}\label{ineq:est-tilde-w}
	\|\tilde{w}(x,t) - \bar{w} h_0(\lambda\tau t+\lambda \xi \cdot x) \phi(x,t) \|_{L^\infty(\mcO)} \leq C(\bar{w},\phi,h_0)\lambda^{-1}.
	\end{equation}
	Indeed, it follows from Lemma \ref{Lemma:OpR} that
	\begin{equation*}
	\nabla \cdot \n = \nabla \cdot (\n' + \n'')= \nabla \cdot \n' - \nabla \cdot \nabla \Delta^{-1} \nabla \cdot \n'=0,
	\end{equation*}
	and
	\begin{equation*}
	\nabla \cdot \V =  \nabla \cdot \V' + \nabla \cdot \mathcal{R}[\mb{B} \n - \partial_t \n - \nabla \cdot \V'] = \mb{B} \n - \partial_t \n.
	\end{equation*}
	Therefore, $(\n,\V)$ solves the linear system $\mathscr{L}(\n,\V)=0$.
	
	Direct computations show that
	\begin{equation}\label{eq:expn'}
	\begin{aligned}
	\n' &= \bar{\n}h_0(\lambda \tau t +\lambda \xi \cdot x) \phi + \lambda^{-6}\bar{\n}\sum_{|\beta|\geq 1, |\alpha+\beta|=6} C_{\alpha,\beta} \ptl_x^\alpha\left(h_6(\lambda \tau t +\lambda \xi \cdot x)\right)\ptl_x^\beta \phi, \\
	&=\bar{\n}h_0(\lambda \tau t +\lambda \xi \cdot x) \phi + \lambda^{-6}\bar{\n}\sum_{|\beta|\geq 1, |\alpha+\beta|=6}\lambda^{|\alpha|} C_{\alpha,\beta} \xi^\alpha h_{6-|\alpha|}(\lambda \tau t +\lambda \xi \cdot x)\ptl_x^\beta \phi
	\end{aligned}
	\end{equation}
	and
	\begin{align}
	\V' =\bar{\V}h_0(\lambda \tau t +\lambda \xi \cdot x) \phi +  \lambda^{-6}\bar{\V}\sum_{|\beta|\geq 1, |\alpha+\beta|=6} \lambda^{|\alpha|} C_{\alpha,\beta} \xi^\alpha h_{6-|\alpha|}(\lambda \tau t +\lambda \xi \cdot x)\ptl_x^\beta \phi.\label{eq:expV'}
	\end{align}
	It follows from \eqref{eq:defn'V'} that $\supp (\n',\V')\subset \supp \phi$, and $\int_{\R^n}(\n',\V')dx=0$.
	The expressions \eqref{eq:expn'} and \eqref{eq:expV'} imply that
	\begin{equation}\label{estnVb}
	\|(\n',\V')-(\bar{\n},\bar{\V})h_0(\lambda \tau t +\lambda \xi \cdot x) \phi\|_{L^\infty(\mcO)}\leq C(|\bar{\n}|,|\bar{\V}|,\phi,h_0)\lambda^{-1}.
	\end{equation}
	
	
	It follows from the property $\bar{\n} \cdot \xi = 0$ that
	\begin{align*}
	\n'' = -\lambda^{-6}\nabla \Delta^{-1} \Delta^3 \nabla \cdot  [\bar{\n}h_6(\lambda \tau t +\lambda \xi \cdot x) \phi]
	=-\lambda^{-6} \nabla \Delta^2    [(\bar{\n} \cdot \nabla \phi)h_6(\lambda \tau t +\lambda \xi \cdot x) ].
	\end{align*}
	Hence it holds that $\supp \n''\subset \supp \phi, \int_{\R^n} \n'' dx = 0$ and 
	\begin{align}\label{inftynpp}
	\|\n''\|_{L^\infty(\mcO)}\leq  C(|\bar{\n}|,\phi,h_0)\lambda^{-1}.
	\end{align}
	
	Since $\mb{B}$ commutes with the Laplacian  $\Delta$, direct  computations  yield that
	\begin{align*}
	\V''&=-\mathcal{R}[\partial_t \n' + \nabla \cdot \V' + \partial_t \n'' - \mb{B} \n]\\
	&=-\lambda^{-6}\mathcal{R}\{ \Delta^3 [\partial_t  \left(\bar{\n}h_6\left( \lambda \tau t +\lambda \xi \cdot x\right)  \phi \right) + \nabla \cdot \left(\bar{\V}h_6\left( \lambda \tau t +\lambda \xi \cdot x\right)  \phi \right)    ]\\
	& \qquad -  \partial_t \nabla \Delta^2 [(\nabla \phi \cdot \bar{\n} )h_6\left( \lambda \tau t +\lambda \xi \cdot x\right)] \\
	& \qquad -  \mb{B} \Delta^3 [\bar{\n}h_6(\lambda \tau t +\lambda \xi \cdot x) \phi] + \mb{B}\nabla \Delta^2    [(\nabla \phi \cdot \bar{\n} )h_6(\lambda \tau t +\lambda \xi \cdot x) ]\}\\
	&=-\lambda^{-6}\mathcal{R}\{ \Delta^3 [\left(\bar{\n} \partial_t \phi + \bar{\V} \cdot \nabla \phi \right)h_6( \lambda \tau t +\lambda \xi \cdot x)        ]\\
	& \qquad - \Delta^2 \nabla \partial_t [(\nabla \phi \cdot \bar{\n} )h_6\left( \lambda \tau t +\lambda \xi \cdot x\right)] \\
	& \qquad - \Delta^2 \mb{B} \left[\Delta (\bar{\n}h_6(\lambda \tau t +\lambda \xi \cdot x) \phi) - \nabla     ((\nabla \phi \cdot \bar{\n} )h_6(\lambda \tau t +\lambda \xi \cdot x) )\right]\},
	\end{align*}
	where $\tau \bar{\n}+ \bar{\V} \cdot \xi = 0$ has been used in the last equality. Hence $V''$ can be  written as
	\begin{align*}
	\V''&=-\lambda^{-6}\mathcal{R} [\Delta^2 \mb{f}],
	\end{align*}
	where
	\begin{align*}
	\mb{f}&=\Delta [\left(\bar{\n} \partial_t \phi + \bar{\V} \cdot \nabla \phi \right)h_6( \lambda \tau t +\lambda \xi \cdot x)        ] -  \nabla \partial_t [(\nabla \phi \cdot \bar{\n} )h_6\left( \lambda \tau t +\lambda \xi \cdot x\right)] \\
	&\qquad -  \mb{B} \{\Delta [\bar{\n}h_6(\lambda \tau t +\lambda \xi \cdot x) \phi] - \nabla     [(\nabla \phi \cdot \bar{\n} )h_6(\lambda \tau t +\lambda \xi \cdot x) ]\}.
	\end{align*}
	It follows from Lemma \ref{Lemma:OpR} that $\mathcal{R}[\Delta^2 \mb{f}]$ is a linear combination of third order derivatives of $\mb{f}$. Therefore,  $\supp \V'' \subset \supp \mb{f}\subset  \supp \phi$, and $\int_{\R^n} \V'' dx =0$.
	Combining the above computations yields \eqref{eq:Lemma6-P1} and that $\supp \tilde{w} \subset \supp \phi$. Furthermore, it holds that
	\begin{align*}
	\|\V''\|_{L^\infty(\mcO)}=\lambda^{-6}\|\mathcal{R} [\Delta^2 \mb{f}]\|_{L^\infty(\mcO)}\leq  C(|\bar{\n}|,|\bar{\V}|,\phi,h_0)\lambda^{-1},
	\end{align*}
	which together with \eqref{estnVb} and \eqref{inftynpp} leads to \eqref{ineq:est-tilde-w}.  
	
	Since for $(x,t) \in \mcO, \bar{w}(h_0 \phi)(x,t) \in [w_1,w_2]$, it follows from \eqref{ineq:est-tilde-w} that 
	\begin{align*}
	\text{dist}(w+\tilde{w},[w_1,w_2])&\leq \text{dist}(w+\bar{w} h_0 \phi,[w_1,w_2]) + |\tilde{w}- \bar{w} h_0(\lambda\tau t+\lambda \xi \cdot x) \phi| \\
	&\leq  C(\bar{w},\phi,h_0)\lambda^{-1}.
	\end{align*}
	Define the disjoint open sets $\mcO_i$ as
	\[\mcO_i = \left\lbrace(x,t) \in \mcO: |w + \bar{w} h_0(\lambda\tau t+\lambda \xi \cdot x) \phi(x,t) - w_i| < \min\left(\frac{\eps}{2}, \frac{|w_2-w_1|}{4}\right)\right\rbrace.\] 
	It follows from \eqref{phi} and \eqref{h_0} that $\mcH^{n+1}(\mcO_i)$ can be arbitrary close to  $\mu_i\mcH^{n+1}(\mcO)$  for $\delta$   sufficiently small and $\lambda$ sufficiently large. For $(x,t) \in \mcO_i$, it holds that
	\begin{align*}
	|w+\tilde{w}(x,t)-w_i| &\leq |w + \bar{w} h_0(\lambda\tau t+\lambda \xi \cdot x) \phi - w_i| + |\tilde{w}- \bar{w} h_0(\lambda\tau t+\lambda \xi \cdot x) \phi| \\
	&\leq \frac{\eps}{2} + C(\bar{w},\phi,h_0)\lambda^{-1}.
	\end{align*}
	Therefore Properties (2) and (3) follow by first choosing $\delta$   sufficiently small and then $\lambda$ sufficiently large. The proof of the proposition is finished.
\end{proof}

\begin{proof}[Proof of Lemma \ref{Lemma:OpR}]
	Given $\mb{f}\in C_c^\infty(\R^n; \R^n)$,
	define $\mb{g}=\Delta^{-1}\mb{f}$ where $\Delta^{-1}u=\mcN*u$ with the Newtonian potential $\mcN$ defined by
	\begin{equation*}
	\mcN(x)=\left\{
	\begin{aligned}
	&\frac{1}{\Gamma(n)}\frac{1}{|x|^{n-2}}\,\, \text{for}\,\, n\geq 3,\\
	&\frac{1}{2\pi}\ln |x|\,\,\quad \,\, \text{for}\,\, n=2.
	\end{aligned}
	\right.
	\end{equation*}
	Set
	\begin{gather*}
	\mathcal{R}[\mb{f}] = \frac{n-2}{2(n-1)} [\nabla \mbP\mb{g} + (\nabla \mbP\mb{g})^t] + \frac{n}{2(n-1)} [\nabla \mb{g} + (\nabla \mb{g})^t] - \frac{1}{n-1} (\nabla \cdot \mb{g})\mb{I},
	\end{gather*}
	where
	\begin{align*}
	\mbP\mb{g} := \mb{g}-\nabla \Delta^{-1}\nabla \cdot \mb{g}.
	\end{align*}
	Then one can verify the properties of $\mathcal{R}$ by straightforward computations. First, one has
	\begin{align*}
	\nabla \cdot \mathcal{R}[\mb{f}]&=\frac{n-2}{2(n-1)} [\Delta \mbP\mb{g} + \nabla \div \mbP\mb{g}] + \frac{n}{2(n-1)} [\Delta \mb{g} + \nabla \div \mb{g}] - \frac{1}{n-1} (\nabla \div \mb{g})\\
	&=\frac{n-2}{2(n-1)} [\Delta \mb{g} - \nabla \div \mb{g} + 0] + \frac{n}{2(n-1)} [\Delta \mb{g} + \nabla \div \mb{g}] - \frac{1}{n-1} (\nabla \div \mb{g})\\
	&=\Delta \mb{g} = \Delta \Delta^{-1}\mb{f} = \mb{f}.
	\end{align*}
	It is clear that $\mathcal{R}[\mb{f}]$ is a symmetric matrix and linear in $\mb{f}$. Furthermore,
	\begin{align*}
	\tr \mathcal{R}[\mb{f}] &= \frac{n-2}{2(n-1)} [2 \nabla \cdot  \mbP\mb{g}] + \frac{n}{2(n-1)} [2 \nabla \cdot \mb{g}] - \frac{n}{n-1} (\nabla \cdot \mb{g})\\
	&=0+\frac{n}{n-1} (\nabla \cdot \mb{g})-\frac{n}{n-1} (\nabla \cdot \mb{g})=0.
	\end{align*}
	Therefore, $\mathcal{R}$ maps $C_c^\infty(\R^n;\R^n)$ to $C^\infty(\R^n; \S_0^{n\times n})$. This proves the existence of $\mathcal{R}$ and the property (1).
	
	Since
	\[\mbP\Delta \mb{f}= \Delta \mb{f} - \nabla \div \mb{f},\]
	one has
	\begin{equation}\label{eq:R_DeltaSquare_h}
	\begin{aligned}
	\mathcal{R}[\Delta^2 \mb{f}]&=\frac{n-2}{2(n-1)} [\nabla (\Delta \mb{f} - \nabla \div \mb{f}) + (\nabla (\Delta \mb{f} - \nabla \div \mb{f}))^t] \\
	& \qquad  + \frac{n}{2(n-1)} [\nabla \Delta \mb{f} + (\nabla \Delta \mb{f})^t] - \frac{1}{n-1} ( \Delta \div \mb{f}).
	\end{aligned}
	\end{equation}
	This implies the second property of $\mathcal{R}$.
	
	It follows from \eqref{eq:R_DeltaSquare_h}  that $\supp \mathcal{R}[\Delta^2 \mb{f}] \subset \supp \mb{f}$. If  $\mb{f} \in C_c^\infty(\R^n;\R^n)$, then taking integral for \eqref{eq:R_DeltaSquare_h} on $\R^n$ yields
	\begin{align*}
	\int_{\R^n} \mathcal{R}[\Delta^2 \mb{f}]dx=0.
	\end{align*}
	Hence one has the third property of $\mathcal{R}$.
	
	Note that
	\[
	\|\mb{f}\|_{L^p(\R^n)}\leq \max(\|\mb{f}\|_{L^1(\R^n)},\|\mb{f}\|_{L^\infty(\R^n)})\quad \text{for}\,\, 1\leq p\leq \infty.
	\]
	Thus the fourth property of $\mathcal{R}$ is then an easy consequence of the boundedness of the Riesz operators and Sobolev embeddings \cite{Stein70}. The proof of the lemma is completed.
\end{proof}

\section{Convex Integration}\label{secconint}

In this section, weak solutions are obtained by iterating strict subsolutions, based on the method of convex integration. The main goal is to show the following proposition.

\begin{Proposition}\label{Prop:ExSol}
	If $(\rho,\underline{\m},\underline{\U},q)$ is  a strict subsolution in $\mathcal{D}$, then there exist infinitely many pairs $(\m,\U)$ such that $(\rho,\m,\U,q)$ are subsolutions and
	\begin{equation}
	(\m,\U)(x,t)\in K_{(x,t)}  \text{ for a.e. }  (x,t) \text{ in }\mathcal{D}, \quad \text{ and } \quad \supp(\m-\underline{\m},\U-\underline{\U}) \subset \overline
	{\mcD}, \label{eq:mU}
	\end{equation}	
	where $K_{(x,t)}$ is the compact set associated with  the strict subsolution $(\rho,\underline{\m},\underline{\U},q)$ in Definition \ref{def:subsolu}.
\end{Proposition}
 To obtain weak solutions, we use a convex integration scheme  to  construct  strongly convergent sequences of strict subsolutions. This gives a more constructive proof. The scheme is similar to the construction in \cite{MS03,Choffrut}, and the convergence  relies on the method of controlled weak convergence in \cite{Sz2}. An inductive argument using Lemma \ref{Lemma:LPW} yields the following lemma.
\begin{Lemma}\label{Lemma:BB}
	Suppose that $\tilde{K}\subset K_{\bar{\rho},\bar{q}}$, and $\bar{w} \in \intco \tilde{K}$ is a constant vector. For any $\varepsilon>0$,  there exists a $\tilde{w} \in  C_c^\infty(Q_1;\R^n \times \S_0^{n \times n})$ such that
	\begin{enumerate}
		\item $\tilde{w}$ satisfies \eqref{eq:Lemma6-P1};
		\item $\bar{w}+\overline{\text{image}(\tilde{w})} \subset  \intco \tilde{K}$, where $\text{image}(\tilde{w})=\tilde{w}(Q_1)$ as defined in \eqref{def:image} and $\overline{\text{image}(\tilde{w})}$ denotes its closure in $\R^n \times \S_0^{n \times n}$;
		\item the following integral estimate holds
		\begin{equation}
		\int_{Q_1} \text{dist}(\bar{w}+\tilde{w}(z),\tilde{K})dz<\varepsilon.
		\end{equation}
	\end{enumerate}
\end{Lemma}
\begin{proof}
	It follows from Lemma \ref{Lemma:FSC} that there exists a finite set $\{w_i\}_{i=1}^N \subset \tilde{K}$ such that $\bar{w} \in \intco \{w_i\}_{i=1}^N$. For $\beta \in (0,1)$, define 
	\[ L_\beta = \{w_i^{(\beta)} = \bar{w} + (1-\beta)(w_i - \bar{w})\}_{i=1}^N. \] It is easy to see that  for any $0<\beta_2<\beta_1<1$,
	\[ \overline{\text{ conv } L_{\beta_1}}  \subset \intco L_{\beta_2} \subset \intco \tilde{K}.\] 
	Observe that  $ \bar{w} \in \intco L_{\beta}$ for any $\beta \in (0,1)$.   Choose $\delta$  such that $\delta\max_i|w_i-\bar{w}|< \eps/4$.
	Set $L^{(0)} = L_{\delta}$ and
	\[  L^{(j+1)}=L^{(j)} \cup \{\nu_1 w_1' + \nu_2 w_2': w_i' \in L^{(j)}, w_2'-w_1' \in \Lambda, \nu_i \in (0,1), \nu_1 + \nu_2 = 1\}.\]
	 Lemma \ref{Lemma:AD} implies that 
	 \[ w_i^{(\delta)} - w_j^{(\delta)}= (1-\delta)(w_i-w_j) \in \Lambda. \] 
	 It follows from Carath\'{e}odory's theorem that 
	 \[   \bigcup_{j=0}^{N_n}L^{(j)} = \text{conv }L_{\delta}, \] where $N_n = \dim(\R^n \times \S_0^{n \times n})$. 
	
	 Let $\tau_j = 2^{-j-1}$ for $j=0,1,2,\cdots$.
	 If $w \in L^{(0)}=L_{\delta}$, set $\tilde{w}=0$. Clearly  $w \in L_{\delta} \subset \intco L_{\delta/2}$. Observe that $w=w_l^{\delta}$ for some $l$ and thus
	 \[ \text{dist}(w,\tilde{K})\leq \max_i |w_i - w_i^{\delta}| \leq \delta\max_i|w_i-\bar{w}|< \eps/4< (1-\tau_0)\eps. \]
	 Assume inductively that  for $0\leq j \leq k$, if $w \in  L^{(j)}$ and $\mcO \subset Q_1$ is any open set, there exists a $\tilde{w} \in C_c^\infty(\mcO)$ satisfying \eqref{eq:Lemma6-P1} and
	 \begin{align}
	 w+\overline{\text{image}(\tilde{w})} \subset  \intco L_{\tau_j\delta}, \quad \frac{1}{|\mcO|}\int_{\mcO} \text{dist}(w+\tilde{w}(z),\tilde{K})dz< (1-\tau_j)\eps. \label{eq:w_k}
	 \end{align} 
	 If $w \in (L^{(k+1)} \setminus L^{(k)})$,  then there exist $w_1', w_2' \in L^{(k)}$ such that
	\[ w = \nu_1 w_1' + \nu_2 w_2', \quad w_2'-w_1' \in \Lambda,\,\, \nu_i \in (0,1),\,\, \nu_1 + \nu_2 = 1. \]
	Given an open set $\mcO \subset Q_1$. Let $\eps_0$ be a small positive constant to be chosen later. In view of Lemma \ref{Lemma:LPW} there exists a $\tilde{w}_0 \in C_c^\infty(\mcO)$ satisfying  \eqref{eq:Lemma6-P1} and two disjoint open sets $\mcO_i \subset \mcO$ such that
	\[\text{dist}(w+\tilde{w}_0,[w_1',w_2'])<\eps_0 \text{ in } \mcO,\]
	and for $i=1,2$,
	\[ |w+\tilde{w}_0-w_i'| < \eps_0 \text{ in } \mcO_i, \quad |\mcH^{n+1}(\mcO_i) - \nu_i\mcH^{n+1}(\mcO)| < \eps_0. \]
	  Since $w_i' \in L^{(k)} \subset \text{conv } L_{\delta} \subset \intco L_{\tau_k \delta}$, it follows from the inductive assumption that  there exists a $\tilde{w}_i \in C_c^\infty(\mcO_i)$ satisfying  \eqref{eq:Lemma6-P1} and
	  \begin{equation}
	  w_i' + \overline{\text{image}(\tilde{w}_i)} \subset \intco L_{\tau_k \delta}, \quad \frac{1}{|\mcO_i|}\int_{\mcO_i} \text{dist}(w_i'+\tilde{w}_i(z),\tilde{K})dz <(1-\tau_k)\eps. \label{ineq:O_i}
	  \end{equation}
	 Let $\tilde{w}=\tilde{w}_0+\tilde{w}_1 + \tilde{w}_2$. Clearly $\tilde{w} \in C_c^\infty(\mcO)$ and satisfies  \eqref{eq:Lemma6-P1}. Furthermore, for $z \in \mcO \setminus (\mcO_1 \cup \mcO_2)$,
	 \begin{align*}
	 w+\tilde{w}(z)&=w+\tilde{w}_{0}(z) \in B_{\eps_0} + [w_1',w_2'] \subset B_{\eps_0} + \text{conv } L_{\tau_k \delta},
	 \end{align*}
	 and for $z\in \mcO_i, i=1,2$,
	 \begin{align*}
	 w+\tilde{w}(z)&=(w+\tilde{w}_0(z)-w_i')+(w_i'+\tilde{w_i}(z)) \in B_{\eps_0} + \text{conv } L_{\tau_k \delta}.
	 \end{align*}
	 Since $\overline{\text{conv }  L_{\tau_k\delta}} \subset \intco L_{\tau_{k+1}\delta}$, for  $\eps_0$ sufficiently small it holds that
	 \begin{equation}
	 B_{\eps_0} + \text{conv } L_{\tau_k \delta} \subset  \intco L_{\tau_{k+1}\delta}. \label{eq:eps_0-1}
	 \end{equation}  
	 Then $w+\overline{\text{image}(\tilde{w})} \subset  \intco L_{\tau_{k+1}\delta}$. 
	 
	 Notice that for $z \in \mcO_i, w+\tilde{w}(z)=w+\tilde{w}_0(z)+\tilde{w}_i(z)$. Thus for $z \in \mcO_i$,
	 \begin{align*}
	 \text{dist}(w+\tilde{w}(z),\tilde{K}) &\leq |w+\tilde{w}_0(z)-w_i'| + \text{dist}(w_i'+\tilde{w}_i(z),\tilde{K})\\
	 &\leq \eps_0 + \text{dist}(w_ i'+\tilde{w_i}(z),\tilde{K}).
	 \end{align*}
	Recalling the bounds \eqref{ineq:Kbdd} and that $\mcO_1$ and $\mcO_2$ are disjoint open subsets of $\mcO$ satisfying 
	\begin{align*}
	\mcH^{n+1}(\mcO \setminus (\mcO_1 \cup \mcO_2)) & = \mcH^{n+1}(\mcO) - \sum_{i=1}^{2} \mcH^{n+1}(\mcO_i)  \\
	&\leq \sum_{i=1}^{2}|\nu_i\mcH^{n+1}(\mcO) -\mcH^{n+1}(\mcO_i)|\\
	& < 2\eps_0,
	\end{align*} one can get that
	\begin{align*}
	\int_{\mcO} \text{dist}(w+\tilde{w}(z),\tilde{K})dz &\leq \sum_{i=1}^2 \int_{\mcO_i}\text{dist}(w+\tilde{w}(z),\tilde{K})dz + 2C(\bar{\rho},\bar{q})|\mcO \setminus (\cup_i\mcO_i )| \\
	& \leq \sum_{i=1}^2 \int_{\mcO_i}\eps_0 + \text{dist}(w_ i'+\tilde{w_i}(z),\tilde{K})dz + 4C(\bar{\rho},\bar{q})\eps_0  \\
	& \leq (|\mcO_1| + |\mcO_2|)(1-\tau_k)\eps+ C \eps_0,
	\end{align*}
	where  \eqref{ineq:O_i} is used in the last inequality. 
	Since $\tau_{k+1}<\tau_k$,  it holds  that 
	\begin{equation}
	\frac{1}{|\mcO|}\int_{\mcO} \text{dist}(w+\tilde{w}(z),\tilde{K})dz \leq (1-\tau_k)\eps+\frac{C\eps_0}{|\mcO|}< (1-\tau_{k+1})\eps,\label{eq:eps_0-2}
	\end{equation} 
	for $\eps_0$  sufficiently small.
	One can choose $\eps_0$ depending on $\mcH^{n+1}(\mcO),\eps,\tau_k$ and $\tau_{k+1}$ so that it satisfies \eqref{eq:eps_0-1} and \eqref{eq:eps_0-2}. The induction step is completed. This finishes the proof of the lemma.
\end{proof}

To deal with variable constrained sets $K_{(x,t)}$, one needs a stability result.

\begin{Lemma}(\cite[Lemma 1]{CGSW14})\label{Lemma:Stability}
	Let $K$ be a compact set. For any compact set $\mcC \subset \intco K$ there exists a $\delta >0$ depending on $\mcC$ and $K$ such that for any compact set $L$ with 
	\[d_H(K,L) < \delta \quad \text{(in Hausdorff distance)}\]
	it holds that
	\[ \mcC \subset \intco L.\]
\end{Lemma}
In view of the above two lemmas, the following key perturbation property holds.
\begin{Lemma}\label{Lemma:OneStage}
	Suppose that $(\rho,w,q)$ is  a strict subsolution in a bounded open set $\mathcal{D}$.  Given  $\varepsilon>0$, there exists a compact set $\mcC\subset \mcD$ and a   sequence $\{w_k\}$ such that $(\rho,w_k,q)$ are strict subsolutions  in $\mathcal{D}$ and $w_k-w \in C_c^\infty(\mcD), \text{supp}(w_k - w) \subset \mcC, w_k \to w$ in $CL^\infty_{w*}$. Furthermore, the following estimates hold 
	\begin{equation}
	\int_\mathcal{D} \text{dist}(w_k(x,t),K_{(x,t)}) dxdt \leq \varepsilon. \label{w_k}
	\end{equation}
	
\end{Lemma}

\begin{proof}
	Set $C_0=2\sup_{z \in \mcD} \mathfrak{C}(\rho(z),q(z))$ with $\mathfrak{C}(\rho,q)$ defined in \eqref{ineq:Kbdd}.
	 Approximate $\mathcal{D}$ by a compact subset $\mcC'$ such that $\mcH^{n+1}(\mcD \setminus \mcC')<\frac{\eps}{4C_0}$.
	For $\zeta \in \mathcal{D}$, applying Lemma \ref{Lemma:BB} to $w(\zeta) \in \intco K_{\zeta}$ yields there exists $v \in C_c^\infty(Q_1)$ satisfying \eqref{eq:Lemma6-P1}
	and  
	\begin{equation}
	w(\zeta)+\overline{\text{image}(v)}\subset int \ conv \ K_{\zeta},\quad \int_{Q_1}\text{dist}(w(\zeta)+v(z),K_{\zeta})dz < \frac{\eps}{4|\mcD|}. \label{ineq:z_0}
	\end{equation}  
	It follows from Lemma \ref{Lemma:Stability} and the continuity of $w(z)$ and $K_{z}$ that  there exists a $r(\zeta)>0$ such that  
	\begin{equation}
	w(z)+\overline{\text{image}(v)}\subset int \ conv \ K_{z} \quad \text{ for }z \in Q_{r(\zeta)}(\zeta). \label{K_z}
	\end{equation}
	Furthermore, by shrinking  $r(\zeta)$ if necessary, the following estimates hold
	\begin{equation}
	d_H(K_z,K_{\zeta}) < \frac{\eps}{8|\mcD|} \quad \text{ and } \quad  |w(z)-w(\zeta)|<\frac{\eps}{8|\mcD|} \quad \text{ for }z \in Q_{r(\zeta)}(\zeta). \label{ineq:z}
	\end{equation} 
	Let $\{\mcO^i=Q_{r(\zeta_i)}(\zeta_i)\}_{i=1}^N$ be a finite covering of the compact set $\mcC'$ such that $\cup \overline{\mcO^i} \subset \mcD$, and let $v_i \in C_c^\infty(Q_1)$ be associated with $\zeta_i$. Set $\mcC:=\cup \overline{\mcO^i}$. Set $r_0 = \frac{1}{2}\min\limits_{i} r(\zeta_i)$. 
	
	Apply the Whitney covering lemma \cite{Stein70} to the open sets $\mcO_1$ and $\{\mcO_i \setminus (\bigcup_{j=1}^{i-1} \overline{\mcO_j} )\}_{i=2}^N$, there exist a family of disjoint open cubes $\{\tilde{Q}^{l}\}_{l=1}^\infty$ with $\mathcal{H}^{n+1}(\mcC \setminus (\cup_l \tilde{Q}^l)) = 0$, each $\tilde{Q}^{l}$ lying in some $\mcO^{i(l)}$. For $k=1,2,\cdots$, decomposing the cubes $\tilde{Q}^{l}$ further if necessary, it follows that there exist finitely many disjointed open cubes $\{Q^j=Q_{r_j}(z_j)\}_{j=1}^{J_k}$, where $r_j =r_j^{(k)}, z_j = z_j^{(k)}$ depending on $k$, satisfying
	\[ \max_{1\leq j\leq J_k}r_j < 2^{-k}r_0, \quad \mathcal{H}^{n+1}(\mcC \setminus (\bigcup_{j=1}^{J_k} Q^{j}))<\frac{\eps}{4C_0},  \] 
	and for each $j=1,\cdots, J_k$, $Q^j \subset \mcO^{i(j)}$ for some $i(j)$. 
	 
	Set $a_j(z) = v_{i(j)}(\frac{z-z_j}{r_j}) \in C_c^\infty(Q^{j})$. Let $\tilde{w}_k(z)= \sum_{j=1}^{J_k}a_j(z)$ and set
	\[ w_k(z) = w(z) + \tilde{w}_k(z) = w(z) + \sum_{j=1}^{J_k}a_j(z). \]
	It is clear that $\tilde{w}_k \in C_c^\infty(\mcD)$ and $\text{supp}\tilde{w}_k \subset \cup_j Q^j \subset \mcC$. Since  $\mathscr{L}a_j= 0$, it holds that $(\rho,w_k,q)$ solves \eqref{eq:subsolutonPDE1} and \eqref{eq:subsolutonPDE2}.  It follows from \eqref{K_z} and $ Q^j \subset \mcO^{i(j)}$ that for $z \in Q^j$,
	\[ w_k(z) = w(z) + a_j(z) \in w(z)+ \text{image}(v_{i(j)})\subset int \ conv \ K_{z}. \] 
	For $z \in \mcD  \setminus (\cup_j Q^j)$, $w_k(z) = w(z) \in \intco K_z$. Therefore $(\rho,w_k,q)$ is a strict subsolution in $\mcD$. 
	
	Substituting $z = z_j + r_j z'$, it follows from \eqref{ineq:z_0} and \eqref{ineq:z} that
	\begin{align*}
	\int_{Q^{j}}\text{dist}(w_k(z),K_z)dz &= (r_j)^{n+1} \int_{Q_1}\text{dist}(w(z(z'))+a_j(z(z')),K_{z(z')})dz' 	\\
	& \leq (r_j)^{n+1} \int_{Q_1}\text{dist}(w(\zeta_{i(j)})+v_{i(j)}(z'),K_{\zeta_{i(j)}})+\frac{\eps}{4|D|}dz'\\
	& \leq   (r_j)^{n+1} \frac{\eps}{2|D|} = \frac{|Q^{j}|}{2|D|}\eps.
	\end{align*}
	Thus it holds that
	\begin{align*}
	\int_\mathcal{D} \text{dist}(w_k(z),K_z) dz&= \sum_{j=1}^{J_k} \int_{Q^j}\text{dist}(w_k(z),K_z) dz + \int_{\mcD \setminus (\cup_j Q^j) }\text{dist}(w_k(z),K_z) dz\\
	&\leq |\mcD|\frac{\eps}{2|D|} + C_0 |\mcD \setminus (\cup_j Q^j) |  \leq \frac{\eps}{2} + C_0 |\mcD \setminus \mcC' | \\
	&\leq \eps.
	\end{align*}
	Let $\msQ^{j}$ be the projection of $Q^j$ on $\R^n$ and denote by $z_j=(x_j,t_j)$. Notice that
	\begin{equation}
	a_j \in  C_c^\infty(Q^{j}), \quad \int a_j(x,t)dx=0, \quad \|a_j\|_{L^\infty} \leq 2C_0. \label{a_j}
	\end{equation}
	For $\phi \in C_c^\infty(\R^n)$, it holds that
	\begin{align*}
	\left|\int_{\R^n} \tilde{w}_k \phi dx\right| &= \left|\sum_{j=1}^{J_k}\int_{Q^j} a_j\phi dx\right| = \left|\sum_{j=1}^{J_k}\int_{Q^j}a_j(x,t)(\phi(x)-\phi(x_j)) dx\right| \\
	&\leq 2C_0 |\supp \phi| \max_{ x \in \msQ^{(j)}}|\phi(x)-\phi(x_j)|.
	\end{align*}
	Since $\text{diam}(\msQ^j)\leq \sqrt{n} 2^{-k}r_0$, one has 
	\[ \left|\int \tilde{w}_k\phi\, dx\right|\to 0 \quad \text{uniformly in } t \text{ as } k\to \infty. \]  It follows from density arguments that $\tilde{w}_k \to 0$ in $CL^\infty_{w*}$. This completes the proof of the lemma.
\end{proof}

\begin{proof}[Proof of Proposition \ref{Prop:ExSol}] 
	 Denote $\underline{w}=(\underline{\m},\underline{\U})$. Define
	 \begin{equation*}
	 X^0 = \{w' \in C_{loc}(\bRp;L^\infty_{w*}) : (\rho,w',q) \text{ are strict subsolutions  in } \mcD \text{ and }w'-\underline{w} \in C_c^{\infty}(\mcD)  \}.
	 \end{equation*}  
	 Since $\underline{w} \in X^0$, $X^0$ is non-empty.	It follows from \eqref{ineq:Kbdd} that $X^0$ is a bounded set in $L^\infty(\R^n \times \Rp)$. Let $X$ be the completion of $X^0$ in $C_{loc}(\bRp;L^\infty_{w*}(R^n))$ topology. Observe that $X$ is metrizable due to the Banach-Alaoglu theorem and hence there is a metric $d_X$ such that $(X,d_X)$ forms a complete metric space (see \cite{dLSz2}). It will be shown that the set of $w=(\m,\U)$ satisfying \eqref{eq:mU}  is a dense set in $X$.
	
 Let $\{\mcD_i\}_{i=1}^\infty$ be a family of bounded open subsets of $\mcD$  such that $\mcD_1 \subset \mcD_2 \subset  \cdots, $ and $\mcD = \cup \mcD_i$. In the case that $\mcD$ is bounded, one can take $\mcD_j = \mcD$ for $j=1,2,\cdots$. 
 
 Given $w' \in X^0$ and $\eps_0>0$. Set $w_1=w'$. Assume that for some $k \geq 1$, $w_k \in X^0$ has been obtained. Then $(\rho,w_k,q)$ is a strict subsolution in $\mcD_{k+1}$ and thus $w_k(z) \in \intco K_z$ for $z \in D$. It follows that $\int_{D_1} \text{dist}(w_k(z),K_z)dz > 0$. Applying
	 Lemma \ref{Lemma:OneStage} to $w_k$ with $\mcD_{k+1}$, there exists a sequence  $\{v_i\}$ satisfying that  $(\rho,v_i,q)$ are strict subsolutions  in $\mcD_{k+1}$, with 
	 \[ v_i-w_k \in C_c^\infty(\mcD_{k+1}), \quad v_i \to w_k \text{ in } CL^\infty_{w*} \quad \text{ as } i \to \infty, \] 
	 and obey the estimates 
	\begin{equation}
	\int_{\mcD_{k+1}} \text{dist}(v_i(z),K_z) dz \leq \min\left\{2^{-k},\frac{1}{2}\int_{\mcD_1} \text{dist}(w_{k}(z),K_z) dz\right\}. \label{ineq:dist_wj}
	\end{equation}
	Since $v_i(z) \in \intco K_z$ for $z \in \mcD_{k+1}$  and  $\supp (v_i - w_k) \subset \mcD_{k+1} $, it follows that $v_i(z) \in \intco K_z$ for $z\in \mcD$. It is easy to see  that $(\rho,v_i,q)$ are strict subsolutions  in  $\mcD$. Since $v_i - \underline{w} = (v_i - w_k) + (w_k - \underline{w}) \in C_c^{\infty}(\mcD)$, it holds that $v_i \in X^0$. Since $v_i \to w_k$ in $CL^\infty_{w*}$ as $i \to \infty$, it is easy to see that for $i$ sufficiently large, 
	\begin{equation}
	  \quad \max_{1\leq j \leq k}\left|\int_{D_j} (v_{i}-w_k)w_k dz\right| < \min\left\{2^{-k},\frac{1}{100|\mcD_k|}\left(\int_{\mcD_1} \text{dist}(w_k(z),K_z) dz\right)^2\right\}.\label{ineq:incre}
	\end{equation}
	Recalling that $d_X$ is induced by the topology of $CL^\infty_{w*}$, for  $i$ sufficiently large it holds that
	\begin{equation}
	d_X(v_i,w_k) < 2^{-k}\eps_0. \label{d_X}
	\end{equation}
	Set $w_{k+1} = v_i$ for some large $i$ such that $w_{k+1}$ satisfies \eqref{ineq:incre} and \eqref{d_X}. Thus we obtain a sequence $\{w_k\} \subset X^0$ satisfying \eqref{ineq:dist_wj},  \eqref{ineq:incre}, and \eqref{d_X} with  $v_i$ replaced by $w_{k+1}$ for $k \geq 1$.
	 
	It follows from \eqref{d_X} that $\{w_k\}$ is a Cauchy sequence in $X$ and thus $w_k$ converges to some $w \in X$. We now show that $\{w_k\}$ converges strongly in $L^2(D_j)$, for any fixed $j$.
	
	Fix any $j \geq 1.$ For $k \geq j$, since $\mcD_1 \subset \mcD_j \subset \mcD_{k+1}$, it follows from \eqref{ineq:dist_wj} that 
	\begin{align*}
	\int_{D_j} |w_{k+1}-w_k|dz &\geq \int_{D_j} \text{dist}(w_k,K_z)dz - \int_{D_j} \text{dist}(w_{k+1},K_z)dz \\
	&\geq \int_{D_j} \text{dist}(w_k,K_z)dz - \int_{D_{k+1}} \text{dist}(w_{k+1},K_z)dz \\
	&\geq \frac{1}{2} \int_{D_j} \text{dist}(w_k,K_z)dz.
	\end{align*}
	Thus applying H\"{o}lder's inequality yields 
	\[ \|w_{k+1} - w_{k}\|_{L^2(D_j)}^2 \geq \frac{1}{|\mcD_j|}\left(\int_{D_j} |w_{k+1}-w_k|dz\right)^2 \geq  \frac{1}{4|\mcD_j|}\left(\int_{D_j} |w_{k+1}-w_k|dz\right)^2. \]
	This together with \eqref{ineq:incre}, yields that
	\begin{align*}
	\|w_{k+1}\|_{L^2(D_j)}^2 - \|w_{k}\|_{L^2(D_j)}^2 &= \|w_{k+1} - w_{k}\|_{L^2(D_j)}^2 - 2\int_{D_j} (w_{k}-w_{k+1})w_k dz \geq 0.
	\end{align*}
	Thus $\{\|w_{k}\|_{L^2(D_j)}^2\}_{k=j}^\infty$ is a non-decreasing sequence. Since  $\|w_{k}\|_{L^2(\mcD_j)} \leq C(\rho,q)|\mcD_j|$ is uniformly bounded, $\{\|w_{k}\|_{L^2(D_j)}^2\}_{k=j}^\infty$ is a Cauchy sequence. 
	It follows from \eqref{ineq:incre} that for $k > m \geq j$, the following estimates hold
	\begin{align*}
	\|w_k - w_m\|_{L^2(D_j)}^2 &\leq 2(\|w_k - w_{k-1}\|_{L^2(D_j)}^2+\cdots + \|w_{m+1} - w_m\|_{L^2(D_j)}^2)\\
	& = 2\sum_{l=m}^{k-1}\left(\|w_{l+1}\|_{L^2(D_j)}^2 - \|w_{l}\|_{L^2(D_j)}^2 -  2\int_{\mcD_j} (w_{l+1}-w_l)w_l dz\right) \\
	&\leq 2(\|w_k\|_{L^2(D_j)}^2-\| w_m\|_{L^2(D_j)}^2) + 2^{-m+3}.
	\end{align*}
	Hence $w_k$ converges strongly in $L^2(\mcD_j)$. Then one can conclude from  \eqref{ineq:dist_wj} and \eqref{d_X} that  
	\[ w(z) \in K_{z} \quad \text{a.e. in }\mcD \quad \text{ and } \quad d_X(w,w') \leq \eps_0.\]	  
	Since the scheme holds for any $w' \in X^0$ and $\eps>0$,  the set of $w=(\m,\U)$ satisfying \eqref{eq:mU}  is a dense set in $X$. Since $w \notin X^0$, $\{w_k\}$ is an infinite sequence in $X^0$, and thus a dense set in $X$ is also infinite. The proof of the proposition is finished.
\end{proof}

As in \cite{dLSz2,Chiodaroli,Feireisl14}, it is possible to construct weak solutions with initial data satisfying \eqref{eq:ModifiedSol}, which are useful for constructing admissible solutions.  It is straightforward to extend Definition \ref{def:subsolu} to the cases where the time domain $\bRp$ is replaced by $[-\tau,\infty)$ for some $\tau \geq 0.$

\begin{Corollary}\label{Cor:Sol}
	Suppose that $(\rho,\underline{\m},\underline{\U},q)$ is  a strict subsolution in $\mathcal{D} \subset \R^n \times [-\tau,\infty)$ for some $\tau>0$,  such that $\mathcal{D}_{t=0} \subset \R^n$ is a non-empty open set and $\rho(\cdot,0) = \rho_0$ with
	\begin{align}
		|\underline{\m} |^2(x,0) & = n \rho_0(x)q (x,0),  \quad &\text{ for a.e. } x \in (\R^n  \setminus \mcD_{t=0}), \label{eq:OutsideDI}\\ 
		(\underline{\m},\underline{\U})(x,t) & \in K_{\rho(x,t),q(x,t)}, \quad &\text{ for a.e. } (x,t) \in (\R^n \times \Rp \setminus \mcD). \label{eq:OutsideD}
	\end{align}
	Then there exists an $\m^\diamond \in L^\infty(\R^n)$ with
	\begin{align}
		\m^\diamond(x) &\in (K_{(x,0)}|_{\R^n}),  &\text{ for a.e. } x \in D_{t=0}, \label{eq:FSI}\\
		\text{ and \quad }|\m^\diamond |^2(x)&= n \rho_0(x)q (x,0),  &\text{ for a.e. } x \in \R^n,\label{eq:ModifiedSol}
	\end{align}
	such that the Cauchy problem \eqref{eq:CES} and \eqref{IC} has infinitely many weak solutions $(\rho,\m^\flat)$  satisfying
	\begin{align}
	\m^\flat(x,t) &\in  (K_{(x,t)}|_{\R^n}),   &\text{ for a.e. } (x,t) \in \mathcal{D}, \label{eq:FS}\\
	\text{ and } \quad |\m^\flat|^2(x,t) &= n \rho(x,t) q(x.t), &\text{ for a.e. } (x,t) \in \R^n \times \Rp, \label{eq:MSEnergy}
	\end{align}
	where $K_{(x,t)}|_{\R^n}=\{\n \in \R^n:(\n,\V) \in K_{(x,t)} \text{ for some } \V \in \S_0^{n \times n} \}$.
	
	Furthermore, there exists a sequence of divergence-free vector fields $\{\tilde{\m}_j\}$ such that
	\begin{equation}
	\tilde{\m}_j \in C_c^\infty(\mcD) \quad \text{ and } \quad \underline{\m}+\tilde{\m}_j \to \m^\flat \text{ in } C_{loc}(\bRp;L^\infty_{w*}(\R^n)). \label{m_j}
	\end{equation}
\end{Corollary}
\begin{proof} 	 Denote $\underline{w}=(\underline{\m},\underline{\U})$. Define
	\begin{equation*}
	X^0 = \{w \in C_{loc}([-\tau,\infty);L^\infty_{w*}) : (\rho,w,q) \text{ are strict subsolutions  in } \mcD \text{ and }w - \underline{w} \in C_c^{\infty}(\mcD)  \}.
	\end{equation*}  
	Given $w \in X^0, \eps>0,$ and a bounded open subset $\Omega_0\subset \mcD_{t=0}$. Modifying slightly  the  proof of Lemma \ref{Lemma:OneStage}, one can obtain a sequence of functions $\{w_k\} \subset X^0$, which satisfies $d_X(w_k,w)\to 0$  and
		\begin{equation}
		\supp(w_k-w) \subset \Omega_0 \times (-\eps,\eps), \quad \int_{\Omega_0} \text{dist}(w_k(x,0),K_{(x,0)}) dxdt \leq \varepsilon\label{w_k-t=0}.
		\end{equation}
		For $\xi \in \Omega_0$, it follows from Lemma \ref{Lemma:BB}, Lemma \ref{Lemma:Stability} and the continuity of $w(z)$ and $K_{z}$ that  there exist $v \in C_c^\infty(Q_1)$ and $r(\xi)>0$ such that $v$ satisfies \eqref{eq:Lemma6-P1} and
		\begin{equation}
		 \int_{Q_1}\text{dist}(w(\xi,0)+v(x,t),K_{\xi,0})\,dxdt < \frac{\eps}{4|\Omega_0|}. \label{ineq:z_0'}
		\end{equation}  
		Furthermore, for $z \in Q_{r(\xi)}(\xi,0)$, we have
		\begin{align*}
		w(z)+\overline{\text{image}(v)}\subset int \ conv \ K_{z}, \\
		d_H(K_z,K_{(\xi,0)}) < \frac{\eps}{8|\Omega_0|},\quad  |w(z)-w(\xi,0)|<\frac{\eps}{8|\Omega_0|}.
		\end{align*}
		Since $Q_1 = [-\frac{1}{2},\frac{1}{2}]^n \times [-\frac{1}{2},\frac{1}{2}]$, it follows  from \eqref{ineq:z_0'} that 
		\[ \int_{[-\frac{1}{2},\frac{1}{2}]^n}\text{dist}(w(\xi,0)+v(x,s),K_{(\xi_i,0)})dx < \frac{\eps}{4|\Omega_0|},\]
		for some $s \in [-\frac{1}{2},\frac{1}{2}]$. 
		Let $\mcC'$ be a compact subset of  $\Omega_0$ with $\mathcal{H}^n(\mcD_{t=0} \setminus \mcC')<\frac{\eps}{4C_0}$, where $C_0=2\sup_{ \mcD} C(\rho,q)$ with $C(\rho,q)$ defined in \eqref{ineq:Kbdd}. 
		Denote by $\msQ_{r(\xi)}(\xi)$ the space cube of length $r(\xi)$ centered at $\xi$. 
		Let $\{O^i=\msQ_{r(\xi_i)}(\xi_i)\}_{i=1}^N$ be a finite covering of $\mcC'$ with $v_i \in C_c^\infty(Q_1)$ and $s_i$ associated with $\xi_i$. Set $r_0 = \min_i r(\xi_i)$.
		For $k \in \mathbb{N}$, there exists  $\{\msQ^j = \msQ_{r_j^{(k)}}(x_j^{(k)})\}_{j=1}^{J_k}$ such that $\mathcal{H}^n(\mcC' \setminus (\bigcup_{j=1}^{J_k} \msQ^{j}))<\frac{\eps}{4C_0}$ and
		\[ r_j^{(k)} \leq \min(2^{-k}r_0,2^{-1}\eps), \quad \msQ^j \subset O^{i(j)} \text{ for some }i(j), \text{ for } j=1,\cdots,J_k. \]
		Set 
		\[ a_j(x,t) = v_{i(j)}\left(\frac{x-x_j^{(k)}}{r_j^{(k)}},\frac{t}{r_j^{(k)}}+s_{i(j)}\right) \in C_c^\infty(\msQ_{r_j}(x_j) \times (-\eps,\eps)) \] and let $w_k = w + \sum_{j=1}^{J_k} a_j$. It is easy to verify that $w_k \in X^0$ and that \eqref{w_k-t=0} holds as in the  proof of Lemma \ref{Lemma:OneStage}.
		
		Using \eqref{w_k-t=0} instead of Lemma \ref{Lemma:OneStage}, the proof of Proposition \ref{Prop:ExSol} can be  adapted to obtain $\{w^{(k)}\} \subset X^0$ with $w^{(1)} = (\underline{\m},\underline{\U})$, satisfying the estimates
		\[ w^{(k+1)}=w^{(k)} \text{ for } t \notin  (-2^{-k},2^{-k}), \quad d_X(w^{(k+1)},w^{(k)})<2^{-k},  \]
		and that for any compact subset $\mcC \subset \mcD_{t=0}$, $\{w^{(k)}(\cdot,0)\}$ is a strongly convergent sequence in $L^2(\mcC)$ and  
		\[ \int_{\mcC} \text{dist}(w^{(k)}(x,0),K_{(x,0)})dx \to 0 \text{ as } k \to \infty. \] Let $w'=(\m',\U') := \lim_k w^{(k)} \in X$. Set $\m^\diamond(x)=\m'(x,0)$. It is easy to see that
		\[(\m',\U')(x,0)\in K_{(x,0)} \quad \text{a.e.  in }\mathcal{D}_{t=0}  \quad \text{ and } \quad \supp(\m'-\underline{\m},\U'-\underline{\U}) \subset \overline
		{\mcD}.\]
		 Since $w' = w^{(k)} $ for $ t \notin (-2^{-k},2^{-k})$ for any $k \geq 1$, $(\rho,w',q)$ is a strict subsolution in $D^{t=0}$. 
	  Applying Proposition \ref{Prop:ExSol} to $(\rho,w',q)$ in $D^{t=0}$  yields infinitely many pairs 
	  \begin{equation}
	  (\m^\flat,\U^\flat) \in X,\,\, (\m^\flat,\U^\flat)(\cdot,0)=(\m',\U')(\cdot,0), \,\, (\m^\flat,\U^\flat)(x,t) \in K_{(x,t)}\ a.e. \text{ in } \mcD. \label{Sol}
	  \end{equation}
	It is clear that \eqref{eq:FSI} and \eqref{eq:FS} are satisfied. Since $K_{(x,t)} \subset K_{\rho(x,t),q(x,t)}$, it is easy to see that \eqref{eq:ModifiedSol} and \eqref{eq:MSEnergy} follow from  \eqref{eq:OutsideDI} and \eqref{eq:OutsideD}. 
	
	Since $(\m^\flat,\U^\flat) \in X$, there exist $(\m_j,\U_j) \in X^0$ converging to $(\m^\flat,\U^\flat)$ in $X$. Let $\tilde{\m}_j = \m_j - \underline{\m}$. It follows from the definition of $X^0$ that one has $\tilde{\m}_j \in C_c^\infty(\mcD)$. Since $\m_j$ and $\underline{\m}$ both solve \eqref{eq:subsolutonPDE1}, $\tilde{\m}_j$ are divergence-free vector fields. This finishes the proof of the corollary.
\end{proof}

\section{Proof of the Main Results}\label{secadweak}

\subsection{Constructions of finite-states admissible solutions}\label{subsec:finite-value}
The starting points of the constructions are the piecewise constant stationary strict subsolutions inspired by \cite{dLSz2}. 
\begin{proof}[Proof of Theorem \ref{Thm:FiniteValue}]
	
	Let $\chi$ be a constant satisfying that $\chi \geq \max_i p(\bar{\rho_i})$. Let 
	\[ \bar{q}_i = -p(\bar{\rho_i}) + \chi  \quad \text{ and } \quad I = \{i:\bar{q}_i > 0\}. \]
	For each $i \in I$, Lemma \ref{Lemma:FiniteValue}  yields  a set of $N_n^*$ states $\mathcal{K}_i \subset K_{\bar{\rho}_i,\bar{q}_i}$  such that $(0,0) \in \intco \mathcal{K}_i$. For $i \notin I$, let $\mathcal{K}_i=\{(0,0)\}$.
	Set
	 \begin{equation}
	 q(x,t) = \bar{q}_i, \quad \text{ and } \quad K(x,t) = \mathcal{K}_i, \quad \text{ for }  x \in \Omega_i.
	 \end{equation}
	Let $\mcD=(\bigcup_{i \in I} \Omega_i) \times (-1,\infty)$. It is easy to see that  $(\rho,0,0,q)$ is a  strict subsolution in $\mcD$ with the piecewise constant constrained set $K_{(x,t)}$.
	Corollary \ref{Cor:Sol} yields  infinitely many weak solutions $(\rho,{\m}^\flat)$ to the system \eqref{eq:CES} for some initial data  $(\rho_0,\m^\diamond)$. 
	
	For $i \notin I$, $q = 0$ in $\Omega_i$. Hence it follows from \eqref{eq:ModifiedSol} and \eqref{eq:MSEnergy} that 
	\[ (\rho,\m^\flat)(x,t) = (\brho_i,0) \quad \text{ for a.e. } (x,t) \in \Omega_i \times \bRp.\]
	For $i \in I$,  denote $\mathcal{K}_i = \{\bar{w}^l\}_{l=1}^{N_n^*}$ and $w = (\m^\flat,\dfrac{\m^\flat \otimes \m^\flat}{\rho}-q \mb{I})$. It follows from \eqref{Sol} that  $w \in X$ and $w(x,t) \in \{\bar{w}^l\}_{l=1}^{N_n^*}\, a.e. $ in $\Omega_i \times \bRp$. Hence $w$ has at most $N_n^*$ states in $\Omega_i$.
	Since $0 \in \intco \{\bar{w}^l\}_{l=1}^{N_n^*}$, one can uniquely represent $0$ as a convex combination of $\{\bar{w}^l\}_{l=1}^{N_n^*}$:
	\begin{equation}
	0 = \sum_{l=1}^{N_n^*} \mu_l \bar{w}^l, \quad 0 < \mu_l < 1, \,\, \sum_{l=1}^{N_n^*} \mu_l=1.\label{rep0}
	\end{equation}
	It follows from the proof of Corollary \ref{Cor:Sol} that there exist $\{w_k\}\subset X^0$ with 
	 \[ w_k \to w \text{ in }CL^\infty_{w*}, \,\, w_1 = 0, \,\, w_{k+1}-w_k = \sum_{j=1}^{J_k}a_j^{(k)} \text{ for } k\geq 1, \]
	  where $a_j^{(k)}$ satisfies  $\supp a_j^{(k)}  \subset Q^{j,(k)} \subset \Omega_{l(j)} \times \bRp $ for some $l(j)$ and $\int a_j(x,t) dx = 0$ due to  \eqref{a_j}. Hence $\int_{\Omega_i} w_k(x,t)dx = 0$. For any compact time interval $J \subset \bRp$, it holds that 
	 \[ \int_{\Omega_i \times J} w(x,t)dx dt  = \lim_k \int_{\Omega_i \times J} w_k(x,t)dx dt  = 0.\] 
	 Since $w(x,t) \in \{\bar{w}^l\}_{l=1}^{N_n^*}\, a.e. $ in $\Omega_i \times \bRp$, it holds that
	 \[ 0=\int_{\Omega_i \times J} w(x,t)dx dt = \sum \nu_l \bar{w}^l, \quad \text{ where } \nu_l =\mcH^{n+1}\left(\{(x,t) \in \Omega_i \times J :w(x,t)=\bar{w}^l\}\right).\]
	 It follows from \eqref{rep0} that $\nu_l = \mu_l |\Omega_i \times J| >0$. Hence $(\rho,\m^\flat)$ has exactly $N_n^*$ states in $\Omega_i$.
	
	It remains to consider the energy inequality \eqref{ineq:energy}. It follows from  \eqref{eq:ModifiedSol} and \eqref{eq:MSEnergy} that
	\begin{align*}
	&\int_{0}^{\infty}\int_{\Omega}\left(\mcI(\rho)+ \frac{|\m^\flat|^2}{2\rho}\right)  \partial_t \varphi + \int_{\Omega} \left(\mcI(\rho_0(x))+ \frac{|\m_0(x)|^2}{2\rho_0}\right) \varphi(x,0)dx = 0.
	\end{align*}
	Furthermore denote $P(\rho,q)=\rho^{-1}(\mcI(\rho)+ \frac{n}{2}q+p(\rho))$. It holds that
	\[ \int_{0}^{\infty}\int_{\Omega}\left(\mcI(\rho)+ \frac{|\m^\flat|^2}{2\rho}+p(\rho)\right) \frac{\m^\flat}{\rho} \cdot \nabla \varphi dxdt = \int_{0}^{\infty}\int_{\Omega} P(\rho,q)\m^\flat \cdot \nabla \varphi dx dt. \]
	It follows from \eqref{m_j} that there exist divergence-free vector fields $\tilde{\m}_j \in C_c^\infty(\mcD)$ with $\tilde{\m}_j \to \m^\flat$ in $CL^\infty_{w*}$.  From the definition of $\mcD$, it holds that $\tilde{\m}_j = 0$ on $\ptl \Omega_i$ for $i \in I$ and $\tilde{\m}_j = 0$ in $\Omega_l$ for $l \notin I$. Hence one has
	\[ \int_{0}^{\infty}\int_{\Omega} P(\rho,q)\tilde{\m}_j \cdot \nabla \varphi dx dt = \sum_{i \in I} P(\bar{\rho}_i,\bar{q}_i)\int_{0}^{\infty}\int_{\Omega_i} \tilde{\m}_j \cdot \nabla \varphi dx dt = 0.
	\]
	Therefore, 
	\[ \int_{0}^{\infty}\int_{\Omega} P(\rho,q)\m^\flat \cdot \nabla \varphi dx dt= \lim_j \int_{0}^{\infty}\int_{\Omega} P(\rho,q)\tilde{\m}_j \cdot \nabla \varphi dx dt = 0. \]
	Since $\mb{B}$ is antisymmetric, it follows that $(\rho,\m^\flat)$ satisfies \eqref{ineq:energy} with equality.
%
\end{proof}

%

\subsection{Constructions of admissible solutions with a general source}\label{subsec:general-source}
Given $\rho_0$ satisfying \eqref{ass:smallness}, the acoustic potential $\Psi$ in \cite{Feireisl14} is employed to construct a strict subsolution such that $\rho$  becomes constant in finite time.

\begin{proof}[Proof of Theorem \ref{Thm:GeneralSource}]
	Consider first the  case: $\rho_0 = \rho^\sharp$.
	Let
	$q=e^{-2\beta t}\delta_0$ for some $\delta_0>0$, where $\beta$ is defined in \eqref{eq:Defb}. It follows from \eqref{ineq:SubsolutionNonstrict}  that $(\rho^\sharp, 0, 0, q)$ is a strict subsolution in $\mcD = \T^n \times (-1,\infty)$ with constrain sets $\ K_{\rho^\sharp,q(x,t)}$.  Then Corollary \ref{Cor:Sol} yields  infinitely many weak solutions $(\rho^\sharp, {\m}^\flat)$ to the system \eqref{eq:CES} with initial data $(\rho^\sharp,\m^\diamond)$. Since $\nabla \cdot {\m}^\flat=-\ptl_t \rho^\sharp = 0$, it follows from  \eqref{eq:Defb}, \eqref{eq:ModifiedSol}, and \eqref{eq:MSEnergy}   that
	\begin{align*}
	&\ptl_t \left(\mcI(\rho^\sharp)+\frac{1}{2}\frac{|{\m}^\flat|^2}{\rho^\sharp}\right) + \nabla \cdot \left[\left(\mcI(\rho^\sharp)+\frac{1}{2}\frac{|{\m}^\flat|^2}{\rho^\sharp}\right)\frac{{\m}^\flat}{\rho^\sharp}\right] - \mb{B}\frac{\m^\flat}{\rho^\sharp}\cdot \m^\flat \leq\frac{n}{2} (\ptl_t {q}+2\beta {q}) = 0.
	\end{align*}
	Therefore,  the energy inequality \eqref{ineq:energy} holds for the weak solution $(\rho^\sharp, \m^\flat)$.
	
	Consider now the case that $\rho_0$ is a small perturbation of $\rho^\sharp$ under the assumption \eqref{ass:smallness}.  
	Let $h(t)$ be a smooth cut-off function satisfying that
	\begin{align}\label{eq:Defh}
	h \in C_c^\infty((-1,1)), \quad h(0) = 1, \quad  0\leq h(t)\leq 1, \quad \text{and}\quad  |h'(t)| \leq 2.
	\end{align}
	Define
	\begin{equation}\label{def:rhomq}
	\begin{aligned}
	\rho(x,t)=[1-h(t)] \rho^\sharp +h(t)\rho_0(x), \quad q(x,t) =  p(\rho^\sharp) - p(\rho(x,t)) + \chi(t)
	\end{aligned},
	\end{equation}
	and
	\begin{equation}
	\quad \m(x,t) = h'(t)\nabla \Psi(x),\quad  {\U} = \mathcal{R}_{\T^n}[-h''(t)\nabla \Psi(x)+h'(t)\mb{B}\nabla \Psi(x)],
	\end{equation}
	where $\chi(t)$ is a function to be determined and $\Psi$ is the unique solution in $\mcL(\mathbb{T}^n)$ of the problem
	\begin{equation*}
	\Delta \Psi =\rho^\sharp -\rho_0\quad \text{in}\,\, \mathbb{T}^n.
	\end{equation*}

	It follows standard estimates for the Poisson equation \cite{GT} and Lemma \ref{Lemma:OpRTorus} that
	\begin{equation}
	\begin{aligned}
	\|\Psi\|_{C^{2,\alpha}(\mathbb{T}^n)}\leq  C\eps \quad\text{and}\quad  \|{\U}\|_{C^{1,\alpha}(\mathbb{T}^n)}\leq  C (|h'(t)|+|h''(t)|)\eps.
	\label{ineq:SchauderEstimates}
	\end{aligned}
	\end{equation}
	Furthermore, the  density can be estimated as
	\begin{equation}\label{ineq:rhoest}
	0<\check{\rho}\leq \rho(x,t) \leq \hat{\rho}, \quad
	\|\rho^\sharp  - \rho\|_{L^\infty(\T^n)} \leq h(t)\sup|\nabla \rho_0| \leq h(t)\eps.
	\end{equation}
	The direct computation shows that $(\rho, \m, \U, q)$ solves \eqref{eq:subsolutonPDE1} and \eqref{eq:subsolutonPDE2}.
	One has
	\begin{align*}
	\frac{\m\otimes \m}{\rho}- \U - q \mb{I} =\frac{(h')^2}{\rho} \nabla \Psi \otimes \nabla \Psi  -  {\U} - \left( p(\rho^\sharp ) - p(\rho) + \chi\right)\mb{I}.
	\end{align*}
	It follows from \eqref{ineq:SchauderEstimates} and \eqref{ineq:rhoest} that the inequality \eqref{ineq:SubsolutionNonstrict} holds, provided one can show 
	\begin{align}\label{ineq:chiSubsolution}
	\chi(t)>C(\hat{\rho},\check{\rho})[h'(t)^2\eps+|h'(t)|+|h''(t)|+h(t)]\eps \end{align}
	for $t \in (-\tau_0,\infty)$ with some small $\tau_0>0$. Then  $(\rho,\m,\U,q)$ is a strict subsolution in $\mcD = \T^n\times(-\tau_0,\infty)$ with constrain sets $K_{\rho(x,t),q(x,t)}$. Therefore
		Corollary \ref{Cor:Sol} yields  infinitely many weak solutions $(\rho,{\m}^\flat)$ to the system \eqref{eq:CES} with initial data  $(\rho_0,\m^\diamond)$. It remains to check that these weak solutions are admissible.
			
	Recalling that $\beta$ is defined in \eqref{eq:Defb}, it follows from  \eqref{eq:ModifiedSol} and \eqref{eq:MSEnergy} that
	\begin{equation}\label{ineq:energyChi}
	\begin{aligned}
	& \ptl_t\left(\mcI(\rho)+ \frac{|{\m}^\flat|^2}{2\rho}\right) + \nabla \cdot\left[\left(\mcI(\rho)+\frac{ |{\m}^\flat|^2}{2\rho}+{p(\rho)}\right)\frac{{\m}^\flat}{\rho}\right] -\frac{{\m}^\flat\cdot (\mb{B}{\m}^\flat)}{\rho}\\
	\leq &\frac{n}{2}[\ptl_t \chi(t)-\ptl_tp(\rho) ] +\ptl_t \mcI(\rho) + n\beta[p(\rho^\sharp) - p(\rho(x,t)) + \chi(t)]\\
	&+\nabla \cdot\left[\frac{\mcI(\rho)+p(\rho)+(n/2)(p(\rho^\sharp)-p(\rho))}{\rho} {\m}^\flat\right] + \frac{n}{2}\chi(t)\nabla \cdot \left(\frac{{\m}^\flat}{\rho}\right). 
	\end{aligned}
	\end{equation}
	
	For $t \in [0,\infty)$, since $\supp h \subset (-1,1)$, it follows that  $(\rho,\m,\U,q)=(\rho^\sharp, 0, 0, \chi)$ for $t \geq 1$, thus reducing to the case  of constant density.  Hence if $\chi(1)> 0$, then $\chi(t)$ can be defined on $[1,\infty)$ satisfying \eqref{ineq:chiSubsolution} and the energy inequality \eqref{ineq:energy}.
	
	For $t\in [0,1)$,
	in view of \eqref{ineq:SchauderEstimates}, \eqref{ineq:rhoest} and the estimates
	\begin{align*}
	&|\nabla \cdot {\m}^\flat|  = |\partial_t \rho|\leq |h'(t)|\eps, \quad |p(\rho^\sharp)-p(\rho)| \leq C(\hat{\rho})(\rho^\sharp - \rho) \leq C(\hat{\rho}) \eps,\\
	& |{\m}^\flat| = \sqrt{n\rho [p(\rho^\sharp)-p(\rho)+\chi]}\leq C(\hat{\rho})(\sqrt{\eps}+\sqrt{\chi}),
	\end{align*}
	it is not hard to see that the weak solution $(\rho, \m^\flat)$ satisfies
	\begin{equation}\label{energyineq}
	\ptl_t\left(\mcI(\rho)+ \frac{|{\m}^\flat|^2}{2\rho}\right) + \nabla \cdot\left[\left(\mcI(\rho)+\frac{ |{\m}^\flat|^2}{2\rho}+{p(\rho)}\right)\frac{{\m}^\flat}{\rho}\right] -\frac{{\m}^\flat\cdot (\mb{B}{\m}^\flat)}{\rho}\leq 0,
	\end{equation}
	as long as $\chi$ satisfies
	\begin{equation}\label{ineq:energyChi01}
	\begin{aligned}
	\ptl_t \chi(t)
	&\leq -2\beta \chi-C(\check{\rho},\hat{\rho},\beta)(\chi^{3/2}+\chi+\chi^{1/2}+1)\eps \quad \text{ for } t \in [0,1).
	\end{aligned}
	\end{equation}
	Therefore there exists a $\bar{\eps}$ depending on $\check{\rho},\hat{\rho},$ and $\beta$ such that if $\eps<\bar{\eps}$, it is possible to find a smooth function $\chi(t)$ satisfying \eqref{ineq:chiSubsolution} and \eqref{ineq:energyChi01} on $[0,1]$.  Hence the solution $(\rho, \m^\flat)$ satisfies
	the energy inequality \eqref{ineq:energy}. Observe that for sufficiently small $\eps$, one may take $\delta=\chi(0)$ to be small as well.
	Finally, the estimate \eqref{ineq:decay1} follows from the constructions.
\end{proof}

\subsection{Constructions of admissible solutions with general density}
The aim is to remove the small-oscillation assumption \eqref{ass:smallness} on $\rho_0$ in the case that $\mb{B}$ is anti-symmetric. However, since \eqref{ineq:energyChi01} is a Riccati-type ordinary differential inequality, the assumption  seems essential for the smooth subsolution ansatz  in \cite{Feireisl14}. The key idea is the construction of a non-smooth  subsolution ansatz based on the  ansatz in \cite{Feireisl14}, which transits to piecewise constant states  before the blow-up time.
\begin{proof}[Proof of Theorem \ref{Thm:GeneralDensity}]
	It is assumed that $\rho_0$ is piecewise Lipschitz. Therefore there exists a family of mutually disjoint open sets $\{\Omega_i \}_{i=1}^\infty \subset \R^n$ with $\rho_0 \in W^{1,\infty}(\Omega_i)$, 
	\begin{equation}
	\R^n = \bigcup_{i=1}^\infty \overline{\Omega_i}\quad \text{and}\quad  \quad \mcH^n\left(\bigcup_{i=1}^\infty \ptl \Omega_i\right) = 0,
	\end{equation}
	and there are positive constants $\check{\rho}$, $\hat{\rho}$, $\bar \varrho$ such that 
	\begin{align}\label{eq7}
	0<\check{\rho} \leq \inf \rho_0(x) \leq \sup\rho_0(x) \leq \hat{\rho}<\infty \quad \text{and}\quad  \quad \sup_i \|\nabla \rho_0\|_{L^\infty(\Omega_i)}  \leq \overline{\varrho} < \infty.
	\end{align}

	Let $T\in (0, 1)$ and $\vartheta\in (0, T)$ be two positive constants to be determined.  It follows from the Whitney covering lemma \cite[Theorem 3, p.16]{Stein70} that there exits a countable family of mutually disjoint open cubes $\{\msQ^{(j)}=\msQ_{r_j}(x^{(j)})\}_{j=1}^\infty$ such that
	\begin{align*}
	\bigcup_{i=1}^\infty \Omega_i = \bigcup_{j=1}^\infty \overline{\msQ^{(j)}}, \quad \overline{\msQ^{(j)}}\subset \Omega_{i(j)} \text{ for some } i(j),
	\end{align*}
	and
	\begin{align*}
	r_j \leq \min\left[\text{dist}(\msQ^{(j)},\ptl {\Omega_{i(j)}}),  (1+|x^{(j)}|^2)^{-\frac{ {n}+1}{2}}(1+\overline{\varrho})^{-1}\vartheta \right]\text{ for } j\in \mathbb{N}.
	\end{align*}
	Since every $\msQ^{(j)}$ lies in some $\Omega_{i(j)}$, so $\rho_0 \in W^{1,\infty}(\msQ^{(j)})$. 
	
	Let $\Psi_i$ be the unique solution in $\mcL(\msQ^{(i)})$ to the problem
	\begin{equation*}
	\left\{
	\begin{aligned}
	&\Delta \Psi_i = {\rho}_i^\sharp - \rho_0  &\text{ in } \msQ^{(i)},\\
	& \frac{\ptl \Psi_i}{\ptl \nu} = 0 &\text{ on } \ptl \msQ^{(i)},
	\end{aligned}
	\right.
	\end{equation*}
	where
	\begin{align}
	{\rho}_i^\sharp  = \dfrac{1}{|\msQ^{(i)}|}\int_{\msQ^{(i)}}\rho_0(x) dx.
	\end{align}

	Let $h$ be a smooth function satisfying \eqref{eq:Defh} and define
	\[h_T(t)=h(t/T) \in C_c^\infty((-T,T)).\]
	For $(x, t)\in \msQ^{(i)} \times (-1,\infty)$, set
	\begin{align*}
	\rho_i(x,t) = [1-h_T(t)]{\rho}_i^\sharp + h_T(t) \rho_0(x), \quad \m_i(x,t) = h_T'(t)\nabla \Psi_i(x),
	\end{align*}
	and
	\begin{align*}
	q_i(x,t) = -p(\rho_i(x,t))+\chi(t),
	\end{align*}
	where $\chi(t)$ is a positive function to be determined satisfying that
	\begin{equation}\label{ineq:chiDecay'}
	\chi(t)=\chi(T) > p(\hat{\rho})  \quad \text{ for } t \in [T,\infty).
	\end{equation}
	Recall that $\mcH^n(\R^n \setminus (\cup_i\msQ^{(i)}))=0$, it suffices to define $\rho,\m,q$ on each cube as
	\begin{equation}
	\rho = \rho_i,\,\, \m = \m_i,\,\, q = q_i,\,\, \text{ in } \msQ^{(i)} \times (-1,\infty). \label{Q_i}
	\end{equation}
	Let $\Phi(y)=r_i^{-2}\Psi_i(x^{(i)}+r_i y)$ and $f(y)={\rho}_i^\sharp - \rho_0( x^{(i)}+r_i y)$. Then $\Phi$ solves
	\begin{equation}
	\left\{
	\begin{aligned}
	& \Delta \Phi = f \text{ in } \msQ_1,\\
	&  \frac{\ptl \Phi}{\ptl \nu}  = 0 \text{ on } \ptl \msQ_1.
	\end{aligned}
	\right.
	\end{equation}
	Since
	\[\|f\|_{C^1(\msQ_1)} \leq r_i \|\nabla \rho_0\|_{L^\infty(\msQ^{(i)})} \leq (1+|x^{(j)}|^2)^{-\frac{{n}+1}{2}}\vartheta, \]
	it follows from the standard estimates for the Poisson equation \cite{GT} that
	\[ \|\nabla \Phi\|_{C^{1,\alpha}(\msQ_1)} \leq C (1+|x^{(j)}|^2)^{-\frac{{n}+1}{2}}\vartheta. \]
	Hence it holds that
	\begin{equation}
	\label{ineq:Estimates_m}
	\|\nabla \Psi_i\|_{L^\infty(\msQ^{(i)})}\leq r_i\|\nabla \Phi\|_{L^\infty(\msQ_1)} \leq C (1+|x^{(j)}|^2)^{-({n}+1)}\vartheta^2.
	\end{equation}
	Thus
	\begin{equation*}
	\|\m\|_{L^\infty(\R^n)} = |h_T'(t)|\sup_i \|\nabla \Psi_i\|_{L^\infty(\msQ^{(i)})} \leq C T^{-1}\vartheta^2, \quad 	\|\ptl_t \m\|_{L^\infty(\R^n)} \leq CT^{-2}\vartheta^2.
	\end{equation*}
	Since the cubes $\{\msQ^{j}\}$ have disjoint interiors, so
	\begin{align*}
		&\sum_j \|\nabla \Psi_j\|_{L^1(\msQ^{(j)})} \leq \sum_j \|\nabla \Psi_j\|_{L^\infty(\msQ^{(j)})}|\msQ^{(j)}| \\
		\leq&  
		 \sum_{l=0}^\infty \sum_{2^l \leq |x^{(j)}| \leq 2^{l+1}} \|\nabla \Psi_j\|_{L^\infty}|\msQ^{(j)}| + \sum_{ |x^{(j)}| \leq 1} \|\nabla \Psi_j\|_{L^\infty}|\msQ^{(j)}| \\
		\leq& \sum_{l=0}^\infty (1+|2^l|^2)^{-(n+1)}\vartheta^2|B_{2^{l+1}+1}| + \vartheta^2|B_2| \leq C\vartheta^2 \left(\sum_{l=0}^\infty 2^{-l}+1\right) \leq C \vartheta^2.
	\end{align*}
	Thus
	\[ \|\m\|_{L^1(\R^n)} = |h_T'(t)| \sum_j \|\nabla \Psi_j\|_{L^1(\msQ^{(j)})} = C T^{-1}\vartheta^2,\quad  \|\ptl_t \m\|_{L^1(\R^n)}  \leq C T^{-2}\vartheta^2.\]
	Hence we can define 
	\[\U = \mathcal{R}[-\ptl_t \m+\mb{B}\m].\]
	It follows from Property (4) of Lemma \ref{Lemma:OpR} that
	\begin{equation}\label{ineq:Estimates_U}
	\|\U\|_{C^\alpha(\R^n)}  \leq  C T^{-2}\vartheta^2, 
	\end{equation}
	where the constant $C$ depends only on $\alpha, n$ and the constant matrix $\mb{B}$.
	Since	
	\begin{equation}\label{eq:mbdQi}
	\nu \cdot \m_i = h_T'(t) \nu \cdot \nabla \Psi_i = 0 \quad \text{ on } \ptl \msQ^{(i)},
	\end{equation}
	and that $\ptl_t \rho_i + \nabla \cdot \m_i = 0$ inside $\msQ^{(i)}$, it holds that for $\varphi \in C_c^\infty(\R^n \times (-1,\infty))$
	\[ \int_{-1}^{\infty}\int_{\R^n} \rho \ptl_t \varphi +  \m \cdot \nabla \varphi \, dxdt = \sum_i \int_{-1}^{\infty} \int_{\msQ^{(i)}} \rho_i \ptl_t \varphi +  \m_i \cdot \nabla \varphi \, dx dt = 0. \]
	Furthermore, direct computations yield
	\begin{align*}
	\ptl_t \m + \nabla \cdot \U + \nabla(p+q) - \mb{B} \m = \ptl_t \m + \nabla \cdot \mathcal{R}[-\partial_t\m+\mb{B}\m] + \nabla\chi(t) -  \mb{B} \m=0, 
	\end{align*}
	where $p+q = \chi$ is used.
	It follows that $(\rho,\m,\U,q)$ solves \eqref{eq:subsolutonPDE1} and \eqref{eq:subsolutonPDE2}  in the sense of distribution.
	
	Since $\supp h_T \subset (-T,T)$, denoting $q_i^\sharp = \chi(T)-p(\rho_i^\sharp)$, it holds that 
	\begin{equation}
	(\m,\U) = (0,0) \quad \text{ for } t \in [T,\infty), \quad (\rho,q) = ({\rho}_i^\sharp,q_i^\sharp) \quad \text{ for } (x,t) \in \msQ^{i} \times [T,\infty). \label{afterT}
	\end{equation} 
	Since $\dfrac{\m\otimes \m}{\rho}- \U - q \mb{I} \leq  \left(\dfrac{|\m|^2}{\rho} + |\U| +p(\rho)- \chi(t)\right)\mb{I}$,
	it follows from \eqref{ineq:Estimates_m} and \eqref{ineq:Estimates_U} that  the inequality \eqref{ineq:SubsolutionNonstrict} holds provided that
	\begin{align}\label{ineq:chiSubsolution'}
	\chi(t)>p(\hat{\rho})+ C(\check{\rho},\hat{\rho})  T^{-2}\vartheta^2 1_{(-T,T)}(t),
	\end{align}
	where $1_{(-T,T)}(t)$ is the characteristic  function of $(-T,T)$.
	Suppose that \eqref{ineq:chiSubsolution'} holds in $(-\tau_0,T)$  for some small $\tau_0>0$. Let
	\[ \tilde{\mcD} =  \bigcup_j \msQ^{(j)} \times (-\tau_0, T), \quad \mcD^{(i)} = \msQ^{(i)} \times (T, \infty), \quad \mcD = \tilde{\mcD} \cup \left(\bigcup_i \mcD^{(i)}\right).\]
	For each $i$,  Lemma \ref{Lemma:FiniteValue}  yields  a set of $N_n^*$ states $\mathcal{K}_i\subset K_{{\rho}_i^\sharp,q_i^\sharp}$  such that $(0,0) \in \intco \mathcal{K}_i$. Set 
	\[K_{(x,t)}=\begin{cases}
	K_{\rho(x,t),q(x,t)} &\text{ in } \tilde{\mcD},\\
	\mathcal{K}_i &\text{ in } \mcD^{(i)}.
	\end{cases}  \]
	It follows from  \eqref{afterT} and \eqref{ineq:chiSubsolution'}  that $(\rho,\m,\U,q)$ is a strict subsolution in $\mcD$ with constrain sets $K(x,t)$.
	Therefore Corollary \ref{Cor:Sol} yields  infinitely many weak solutions $(\rho,{\m}^\flat)$ to the system \eqref{eq:CES} with initial data  $(\rho_0,\m^\diamond)$.
	
	 For $t \geq T$, in view of \eqref{afterT}, as in the proof of Theorem \ref{Thm:FiniteValue}, we can show that $(\rho,\m)$ has exactly $N_n^*$ states in $\msQ^{i} \times [T,\infty)$, and that the energy inequality \eqref{ineq:energy} holds with equality in $[T,\infty)$.
	 	
	For $t\in [0,T)$, observe that in view of \eqref{m_j} and \eqref{eq:mbdQi}, it suffices to show that the energy inequality \eqref{ineq:energy} holds in each $\msQ^{(i)}$.
	It follows from \eqref{eq:ModifiedSol} and \eqref{eq:MSEnergy} that
	\[
	\dfrac{|\m^\flat|^2}{2\rho}=\dfrac{n q_i}{2}\quad\text{in}\quad  \msQ^{(i)}.
	\] 
	Recall that $\rho = \rho_i$ and $q = q_i$ in $\msQ^{(i)}$. When $\mb{B}$  is antisymmetric, the energy inequality \eqref{ineq:energy} is reduced to  
	\begin{equation}
	\begin{aligned}
	&\partial _t\left(\mcI(\rho_i)+ \frac{|\m^\flat|^2}{2\rho}\right) + \nabla \cdot \left\{\left(\mcI(\rho_i)+\frac{|\m^\flat|^2}{2\rho}+p(\rho_i)\right) \frac{\m^\flat}{\rho}\right\} \\
	&= \partial _t\left(\mcI(\rho_i)+ \frac{n q_i}{2}\right) + \nabla \cdot \left\{\left(\mcI(\rho_i)+\frac{n q_i}{2}+p(\rho_i)\right) \frac{\m^\flat}{\rho_i}\right\} \leq 0 \quad \text{ in } \msQ^{(i)} \times [0,T). \label{ineq:eQi}
	\end{aligned}
	\end{equation}
	Direct computations yield
	\begin{equation*}\label{ineq:energyQiInteral}
	\begin{aligned}
	 &\partial _t\left(\mcI(\rho_i)+ \frac{n q_i}{2}\right) + \nabla \cdot \left\{\left(\mcI(\rho_i)+\frac{n q_i}{2}+p(\rho_i)\right) \frac{\m^\flat}{\rho_i}\right\} \\
	=&\left( \frac{n}{2}\ptl_t \chi(t)-\frac{n}{2}\ptl_t p(\rho_i)+ \ptl_t\mcI(\rho_i) \right) + \nabla \cdot\left[\left(\mcI(\rho_i)+(1-\frac{n}{2}){p(\rho_i)}\right)\frac{\m^\flat}{\rho_i}\right]\\
	& + \frac{n}{2}\chi(t)\nabla \cdot \left(\frac{\m^\flat}{\rho_i}\right)\\
	= &\frac{n}{2}\left\{\ptl_t \chi(t)-
	\chi(t)\frac{\rho_i \ptl_t \rho_i+\m^\flat\cdot \nabla \rho_i}{\rho_i^2}+\m^\flat \cdot \nabla \left[\frac{2\mcI(\rho_i)+(2-n)p(\rho_i)}{n\rho_i}\right]\right.\\
	&\quad \left.- \left[\frac{2\mcI(\rho_i)+(2-n)p(\rho_i)}{n\rho_i}+p'(\rho_i)-\frac{2}{n}\mcI'(\rho_i)\right]\ptl_t \rho_i\right\},
	\end{aligned}
	\end{equation*}
	where $\nabla \cdot \m^\flat  = -\partial_t \rho_i$ in $\msQ^{(i)}$ is used in the last equality.
	Furthermore, the construction gives
	\begin{align*}
	& |\ptl_t \rho_i|= |h'_T(t)(\rho_0-{\rho}_i^\sharp)|\leq |h'_T(t)|r_i\|\nabla \rho_0\|_{L^\infty(\msQ^{(i)})}\leq C T^{-1}\vartheta,\\
	& |\nabla \rho_i| \leq \|\nabla \rho_0\|_{L^\infty(\msQ^{(i)})} \leq \overline{\varrho}, \quad \check{\rho}\leq \rho_i \leq \hat{\rho}, \quad |\m^\flat| = \sqrt{n\rho_iq_i} \leq (n \hat{\rho})^{1/2}\sqrt{\chi},
	\end{align*}
	therefore, the energy inequality \eqref{ineq:eQi} follows from
	\begin{align}\label{ineq:energyChi'}
	\chi'(t) \leq -C_3 \overline{\varrho} \chi^{3/2} - C_2 \chi - C_1 \overline{\varrho} \chi^{1/2} -C_0,\quad  0 \leq t <T,
	\end{align}
	where the positive constants $C_j=C_j(\check{\rho},\hat{\rho})$ are independent of $\msQ^{(i)}$ and the choices of the constants $T$ and $\vartheta$. Therefore, choosing the constants $T$ and $\vartheta$ to be sufficiently small, we can find a smooth function $\chi$ satisfying \eqref{ineq:chiDecay'}, \eqref{ineq:chiSubsolution'},  and \eqref{ineq:energyChi'} on $[0,\infty)$. It is not hard to see that $T$ can be chosen to be of the order of  $C(\check{\rho},\hat{\rho}) \overline{\varrho}^{-1}$. This finishes the proof of the theorem.
\end{proof}

\bigskip

{\bf Acknowledgement.}
The work is a part of the PhD thesis of the first author. Part of the work was done when the second author visited the Institute of Mathematical Sciences at The Chinese University of Hong Kong and the first author visited Shanghai Jiao Tong University. They thank the both institutions for their support and hospitality.
Xie is
supported in part by NSFC grants 11201297 and 11422105, Shanghai Chenguang
Program,  and the Program for
Professor of Special Appointment (Eastern Scholar) at Shanghai
Institutions of Higher Learning. Luo and Xin are supported in part by Zheng
Ge Ru Foundation, Hong Kong RGC Earmarked Research Grants
CUHK4041/11P, and CUHK4048/13P, a Focus Area Grant from The Chinese
University of Hong Kong, and a CAS-Croucher Joint Grant.

\end{document}